\documentclass[twoside,12pt,a4paper]{article}

\usepackage{amsfonts}
\usepackage{amssymb}
\usepackage{amsmath}
\usepackage{amsthm}
\usepackage{color}
\usepackage{enumerate}

\newtheorem{theorem}{Theorem}[section]

\newtheorem{claim}[theorem]{Claim}

\newtheorem{corollary}[theorem]{Corollary}
\newtheorem{definition}[theorem]{Definition}
\newtheorem{example}[theorem]{Example}

\newtheorem{lemma}[theorem]{Lemma}
\newtheorem{notation}[theorem]{Notation}

\newtheorem{proposition}[theorem]{Proposition}
\newtheorem{question}[theorem]{Question}

\newtheorem{remark}[theorem]{Remark}
\newtheorem*{theorem*}{Theorem}

\def\uh{\upharpoonright}

\DeclareMathOperator{\dom}{dom}

\DeclareMathOperator{\ran}{ran}
\DeclareMathOperator{\stem}{stem}
\DeclareMathOperator{\Ult}{Ult}
\DeclareMathOperator{\cf}{cf}

\DeclareMathOperator{\Ord}{Ord}

\DeclareMathOperator{\ZFC}{ZFC}

\DeclareMathOperator{\GCH}{GCH}

\DeclareMathOperator{\rank}{rank}
\DeclareMathOperator{\Cone}{cone}
\DeclareMathOperator{\crit}{crit}
\DeclareMathOperator{\NL}{NL}

\begin{document}
\title{Prikry-type forcing and minimal $\alpha$-degree}
\author{Yang Sen}

\maketitle
\begin{abstract}
In this paper, we introduce several classes of Prikry-type forcing notions,
two of which are used to produce minimal generic extensions, and the third
is applied in $\alpha$-recursion theory to produce minimal covers. The
first forcing as a warm up yields a minimal generic extension at a measurable
cardinal (in $V$), the second at an $\omega$-limit of measurable cardinals $\langle\gamma_n\colon n<\omega\rangle$  such that each  $\gamma_n$
( $n>0$) carries $\gamma_{n-1}$-many normal measures. Via a notion of $V_\gamma $-degree (see Definition \ref{def:vgammadegree}), we transfer the second Prikry-type construction for minimal generic extensions
to a construction for minimal degrees in $\alpha$-recursion theory. More explicitly,
\begin{theorem*}
Suppose $\langle\gamma_n\colon n<\omega\rangle$  is a strictly increasing sequence of measurable cardinals such that for each $n>0$,  $\gamma_n$ carries at least $\gamma_{n-1}$-many normal measures. Let $\gamma=\sup\{\gamma_n\colon n<\omega\}$. 
Then there is an $A\subset\gamma$ such that
\begin{itemize}
\item[(a)] $(L_{\gamma},\in,A)$ is not admissible.
\item[(b)] The $\gamma$-degree that contains $A$ has a minimal cover.
\end{itemize}
\end{theorem*}
\end{abstract}

\section{Introduction}

Given a $\kappa$-complete ultrafilter $U$ on an infinite cardinal $\kappa$, Prikry forcing notion $\mathbb{P}_U$ is the set of all pairs $(s,A)$, where $s\in[\kappa]^{<\omega}$ and $A\in U$, ordered by $(s,A)\leq(t,B)$ if $t\subset s$, $s\cap(\max(t)+1)=t$ and $s{\setminus}t\subset B$. The basic effect of Prikry forcing is changing the cofinality of $\kappa$ to $\omega$ and preserving all other cofinalities. Since Silver's model for the failure of $\GCH$ at a measurable cardinal (see \cite{2,1}), Prikry forcing became a fundamental tool in forcing construction related to cardinal arithmetic involved large cardinal. Most variations of Prikry like forcings
are discussed in Gitik \cite{18}, which is a comprehensive source for Prikry-type forcings.


In this paper, we introduce three classes of Prikry-type forcings: $\mathbb{P}_{\vec{\mu}}$, $\mathbb{P}_{\mathcal{D}}$ and $\mathbb{Q}_{\mathcal{D}}$.
Most Prikry-type forcings use one measure at one (measurable) cardinal,
the above posets are defined with multiple ultrafilters on each cardinal, which enables us to produce a new structural feature: $\mathbb{P}_{\vec{\mu}}$ and $\mathbb{P}_{\mathcal{D}}$ yield minimal generic extensions;\footnote{\cite{PeterKoepke2012} independently
shows that $\mathbb{P}_{\vec{\mu}}$ yields a minimal generic extension.}
 and as an application in $\alpha$-recursion
theory, $\mathbb{Q}_{\mathcal{D}}$ gives us minimal cover in $\alpha$-degrees.

Assume $P$ is a forcing notion, $G$ is a $P$-generic filter over $V$. By a result of Laver, Woodin, Hamkins(\cite{19}), $V$ is definable in $V[G]$. Then the following are equivalent:
\begin{itemize}
\item[(a)] If $X$ is a set of ordinals in $V[G]$, then $X\in V$ or $V[X]=V[G]$.
\item[(b)] If $M$ is a definable (with parameters) class of $V[G]$, $M$ is an inner model of $\ZFC$ contains $V$, then $M=V$ or $M=V[G]$.
\end{itemize}
%
We say ``\emph{$G$ is minimal over $V$}" or ``\emph{$P$ yields minimal extensions}" if either $(a)$ or $(b)$ holds.
For example, Sacks perfect tree forcing yields minimal extensions, whereas Cohen forcing does not.
It is also a folklore that the classical Prikry forcing $\mathbb{P}_U$ does not yield minimal extensions.

We give two proofs for $\mathbb{P}_{\vec{\mu}}$ and $\mathbb{P}_{\mathcal{D}}$ yielding minimal generic extensions. The first argument is very much combinatorial (see \S 3,\S 4), the second one uses iterated ultrapowers (see \S 5). In the combinatorial method, we use the sum of ultrafilters (see \S 2) to handle forcing conditions, and the computations of the generic filter from a new set for the two posets are different.

In our second argument, we introduce two classes of iterated ultrapowers, $(*)$-iterated ultrapower and diagonal $(*)$-iterated ultrapower, and study the structures of their \emph{intermediate submodels}. Similar to the classical
situation, $\mathbb{P}_{U}$ corresponds to the iterated ultrapower constructed from $U$, our $\mathbb{P}_{\vec{\mu}}$ and $\mathbb{P}_{\mathcal{D}}$ correspond
to relevant $(*)$-iterated ultrapower and diagonal $(*)$-iterated ultrapower respectively. Therefore the iterated ultrapower proofs for $\mathbb{P}_{\vec{\mu}}$ and $\mathbb{P}_{\mathcal{D}}$ are quite similar. This gives us the advantage of studying the posets via the intermediate submodels and vice versa. See the author's doctoral dissertation \cite{Yangsen} for details.

In 1956, Spector showed that
there is a minimal Turing degree. Inspired by this result, we introduces $\mathbb{Q}_{\mathcal{D}}$ (see Section 6), a local version of $\mathbb{P}_{\mathcal{D}}$ used to produce a \emph{minimal cover} in $\alpha$-recursion theoretic sense.

\begin{definition}
Let $\alpha$ be an admissible ordinal and $\mathbf{a},\mathbf{b}$ be $\alpha$-degrees. We say that $\mathbf{a}$ is a \emph{minimal cover} of
$\mathbf{b}$ if $\mathbf{a}>_{\alpha}\mathbf{b}$ and there is no $\alpha$-degree $\mathbf{c}$ such that
$\mathbf{a}>_{\alpha}\mathbf{c}>_{\alpha}\mathbf{b}$. If $\mathbf{a}$ is a minimal cover of $\mathbf{0}$, then we say $\mathbf{a}$ is a \emph{minimal degree}.
\end{definition}

The following minimal cover theorem is a theorem of Shore.

\begin{theorem}[Shore \cite{13}]
If $\alpha$ is a $\Sigma_2$-admissible ordinal, then minimal $\alpha$-degree exists. Suppose $\gamma$ is an admissible ordinal, $A\subset\gamma$ and $(L_{\gamma},\in,A)$ is a $\Sigma_2$-admissible structure, then $A$ has a minimal cover.
\end{theorem}

$(L_{\gamma},\in,A)$ is a \emph{$\Sigma_n$-admissible structure} if for any $\Sigma_n$-formula $\phi(x,y)$ in the language $\{\in,A\}$,
\[(L_{\gamma},\in,A)\models\forall{u}[\forall{x\in u}\exists{y}\phi(x,y)\rightarrow
\exists{z}\forall{x\in{u}}\exists{y\in{z}}\phi(x,y)].\]
$(L_{\gamma},\in,A)$ is \emph{admissible} if it is $\Sigma_1$-admissible.

$\alpha$-recursion theory generalizes classical recursion theory to higher
ordinals. However, as remarked by Simpson in his \cite{15}, ``it is not always easy to appropriately generalize the
statement of a theorem of ordinary degree theory to $\alpha$-degree
theory, much less the proof. One obstacle is that the admissibility
of $\alpha$ does not imply admissibility of the expanded structure
$(L_{\alpha},\in,C)$ where $C\subseteq\alpha$, even if $C$ is
$\alpha$-r.e. and $\alpha$-regular. Therefore `relativization' to
$C$ may be difficult or impossible.'' For example, a long-standing open problem in $\alpha$-recursion theory is whether there is a minimal $\alpha$-degree at $\alpha=\aleph_{\omega}$.
More generally, one can ask

\begin{question}\label{question}
Are there $\gamma$ and $A\subset\gamma$ such that
\begin{enumerate}[$(i)$]
\item For each $n$, $\gamma$ is $\Sigma_n$-admissible.
\item $(L_{\gamma},\in,A)$ is not admissible.
\item The $\gamma$-degree that contains $A$ has a minimal cover.
\end{enumerate}
\end{question}

We will partially answer this question in \S 6 under some mild large cardinal assumption.
Our approach is to use a notion of $V_{\gamma}$-degree (this $\gamma$ is the supremum of $\langle\gamma_n\colon n<\omega\rangle$), which is isomorphic to a cone of $\gamma$-degrees, and via which we can transfer the Prikry-type construction (using $\mathbb{P}_{\mathcal{D}}$) to a construction of a minimal cover of some $\gamma$-degree, i.e. a minimal degree in this cone of $\gamma$-degrees.

\subsection*{Notation}

We use standard set theoretic conventions and notations.
Note that all ultrafilters in this paper are non-principal. Let $U$ be an ultrafilter on $X$, $f_1$ and $f_2$ are functions with domain $X$, we say $f_1$ and $f_2$ are \emph{$U$-equivalent} if $\{x\in X\colon f_1(x)=f_2(x)\}\in U$. We use $f_1\sim_U f_2$ to denote this property. If $U$ is countably complete, then $f_1$ is $U$-equivalent to $f_2$ is just $[f_1]_U=[f_2]_U$ in the ultrapower $\Ult(V,U)$.
Given an infinite cardinal $\kappa$, an ultrafilter $U$ on $\kappa$ is called a \emph{normal measure} if $U$ is $\kappa$-complete and closed under diagonal intersection. $\kappa$ is measurable iff $\kappa$ has a normal measure on it.

In this paper, two extra conditions are given for elementary embeddings $j\colon M\rightarrow N$ between transitive models of $\ZFC$:

\begin{itemize}
\item[1.](Non-trivial) It is not the case that $M=N$ and $j$ is identity.
\item[2.](Cofinal) $j''N\cap\Ord$ is cofinal in $M\cap\Ord$.
\end{itemize}

Definitions and notations in $\alpha$-recursion theory used in this paper all appeared in \cite{11} and \cite{14,15,16}.
Let us fix some concepts and notations to be used in this paper.

\begin{itemize}
\item $J\colon\Ord\times\Ord\rightarrow\Ord$ is a \emph{primitive recursive} bijection such that $x,y\leq J(x,y)$. When $J(s,t)=x$, we write
$(x)_0$ for $s$ and $(x)_1$ for $t$.
\item For a limit ordinal $\alpha$, we fix a bijection $K\colon\alpha\rightarrow L_{\alpha}$ which is $\Sigma_1$-definable in $L_{\alpha}$. We often write $K_{x}$ for $K(x)$. Elements of $L_{\alpha}$ are called \emph{$\alpha$-finite} set, and $K_{x}$ is called the $x$-th $\alpha$-finite set.
\item For $A\subset\alpha$, define $\NL(A)=\{J(x,y)\colon K_x\subset A\ \&\ K_y\subset\alpha{\setminus}A\}$. $\NL(A)$ is also a subset of $\alpha$.
\item Suppose $P,X\subset\alpha$, write $P^X=\{a\in\alpha\colon\exists x\in X\ J(a,x)\in P\}$. For $t\in\alpha$, let $P_t=\{a\in\alpha\colon J(a,t)\in P\}$. Then $P^X=\bigcup_{t\in X}P_t$.
\item For $A,B\subset\alpha$, $A\leq_{\alpha}B$ iff there is a $\alpha$-RE set $P$ such that $\NL(A)=P^{\NL(B)}$. Equivalence classes associated
to this poset form so-called $\alpha$-degrees. Sets in the least $\alpha$-degree are $\alpha$-recursive sets.
\item Suppose $\alpha$ is an ordinal. A set $A\subset\alpha$ is \emph{regular} if for any $K\in L_{\alpha}$, $K\cap A\in L_{\alpha}$. There is a subset of  $\omega_1^{CK}$ which is non-regular, but all subsets of $\omega$ are regular. Non-regularity is a major feature which distinguishes $\alpha$-recursion theory from the classical recursion theory.
\item For $A,B,C\subset\alpha$, set $A\oplus B=\{J(\xi,0)\colon \xi\in A\}\cup\{J(\xi,1)\colon \xi\in B\}$, $A\oplus B\oplus C=\{J(\xi,0)\colon \xi\in A\}\cup \{J(\xi,1)\colon \xi\in B\}\cup \{J(\xi,2)\colon \xi\in C\}$. $A\oplus B$ is contained in the least $\alpha$-degree greater than $\alpha$-degree of $A$ and $B$; $A\oplus B\oplus C$ is contained in the least $\alpha$-degree greater than $\alpha$-degree of $A$, $B$ and $C$.
\end{itemize}

\section{Sum of ultrafilters}

Sum of ultrafilters is not a new concept, it has two effects in this paper: we will use this concept to express some results on ultrafilters; using sum of ultrafilters, handling of Prikry-tree become convenient. \cite{25} gives a comprehensive exhibition of results about ultrafilters.

\begin{definition}\label{def:sum}
$I$ is a non-empty set. $\langle B_i\colon i\in I\rangle$ is a family of non-empty sets. $U$ is a ultrafilter on $I$. $\langle U_i\colon i\in I\rangle$ satisfies for each $i\in I$, $U_i$ is a ultrafilter on $B_i$. $\bigoplus_{i\in I}B_i=\{(i,x)\colon x\in B_i,\ i\in I\}$. Define $U*\langle U_i\colon i\in I\rangle$ is a family of subsets of $\bigoplus_{i\in I}B_i$: for $T\subset\bigoplus_{i\in I}B_i$, \[T\in U*\langle U_i\colon i\in I\rangle\mbox{ iff }\{i\in I\colon\{x\in B_i\colon(i,x)\in T\}\in U_i\}\in U.\]
\end{definition}

Clearly, $U*\langle U_i\colon i\in I\rangle$ is an ultrafilter on $\bigoplus_{i\in I}B_i$. In particular, if for each $i\in I$, $B_i=B$, then $\bigoplus_{i\in I}B_i=I\times B$. In this case, $U*\langle U_i\colon i\in I\rangle$ is an ultrafilter on cartesian product $I\times B$. This is a generalization of product ultrafilter described in \cite{8}: $U'$ and $U''$ are ultrafilters on $I$ and $B$ respectively, then $U'\times U''$ in \cite{8} is just $U'*\langle U_i\colon i\in I\rangle$ where each $U_i=U''$.

\begin{proposition}\label{lambdacomplet}
Suppose $\lambda$ is an infinite cardinal. For each $i\in I$, $U_i$ and $U$ are $\lambda$-complete, then $U*\langle U_i\colon i\in I\rangle$ is $\lambda$-complete.
\end{proposition}

\begin{proof}
$\langle A_{\alpha}\colon\alpha<\theta\rangle$ is a sequence of elements of $U*\langle U_i\colon i\in I\rangle$, $\theta<\lambda$. For $\alpha<\theta$, let $K_{\alpha}=\{i\in I\colon\{x\in B_i\colon (i,x)\in A_{\alpha}\}\}$. Then $K_{\alpha}\in U$. So $K=\bigcap_{\alpha<\theta}K_{\alpha}\in U$. For $i\in K$, $C_i=\bigcap_{\alpha<\theta}(A_{\alpha}\cap(\{i\}\times B_i))$ is in $U_i$. $\bigcup_{i\in K}C_i\in U*\langle U_i\colon i\in I\rangle$ and $\bigcup_{i\in K}C_i\subset\bigcap_{\alpha<\theta}A_{\alpha}$, thus $\bigcap_{\alpha<\theta}A_{\alpha}\in U*\langle U_i\colon i\in I\rangle$.
\end{proof}

Before studying properties of sum of ultrafilters, let us recall some facts about ultrafilters, these facts are useful for here and other parts of this paper.

\begin{proposition}\label{bujiao}
Suppose $\kappa$ is an infinite cardinal, $|C|=\kappa$, $\delta\leq\kappa$, $\langle\mu_{\alpha}\colon\alpha<\delta\rangle$ is a sequence of distinct ultrafilters on $C$ such that for all $\alpha<\theta$,
\begin{itemize}
\item $\mu_{\alpha}$ is $\kappa$-complete.
\item $\langle T_{\beta}:\beta<\kappa\rangle$ is a sequence of elements of $\mu_{\alpha}$, then there is $T\in\mu_{\alpha}$ such that $\forall\beta<\kappa$, $|T{\setminus}T_{\beta}|<\kappa$.
\end{itemize}
Then there is a sequence $\langle A_{\alpha}\colon\alpha<\delta\rangle$ such that for each $\alpha<\delta$, $A_{\alpha}\in\mu_{\alpha}$ and if $\alpha\neq\beta$, then $A_{\alpha}\cap A_{\beta}=\emptyset$.

In particular, if for all $\alpha<\theta$, $\mu_{\alpha}$ is a normal measure on $\kappa$, then the hypothesis above is satisfied, so the conclusion holds.
\end{proposition}

\begin{proof}
Given $\alpha<\theta$, for $\beta<\theta$ such that $\beta\neq\alpha$, let $F_{\beta}\in\mu_{\alpha}$ and $F_{\beta}\notin\mu_{\beta}$ (since $\mu_{\alpha}\neq\mu_{\beta}$, such $F_{\beta}$'s exists). If $\theta<\kappa$, define $X_{\alpha}=\bigcap_{\beta\neq\alpha}F_{\beta}$, by $\kappa$-completeness of $\mu_{\alpha}$, $X_{\alpha}\in\mu_{\alpha}$. If $\theta=\kappa$, let $X_{\alpha}\in\mu_{\alpha}$ such that if $\beta\neq\alpha$, $|X_{\alpha}{\setminus}F_{\beta}|<\kappa$. Also, if $\beta\neq\alpha$, then $X_{\alpha}\notin\mu_{\beta}$. The reason is: if $\theta<\kappa$, then $X_{\alpha}\subset F_{\beta}$ which is not in $\mu_{\beta}$; if $\theta=\kappa$, then $X_{\alpha}$ is a subset of $F_{\beta}$ modulo a set of size less than $\kappa$, so by $\kappa$-completeness of $\mu_{\beta}$, $X_{\alpha}\notin\mu_{\beta}$. Define $A_{\alpha}=X_{\alpha}\cap\bigcap_{\xi<\alpha}(\kappa{\setminus}X_{\xi})$. Then $A_{\alpha}\in\mu_{\alpha}$( since for $\xi<\alpha$, $X_{\xi}\notin\mu_{\alpha}$, so $\kappa{\setminus}X_{\xi}\in\mu_{\alpha}$, so by $\kappa$-completeness of $\mu_{\alpha}$, $\bigcap_{\xi<\alpha}(\kappa{\setminus}X_{\xi})\in\mu_{\alpha}$, So $A_{\alpha}\in\mu_{\alpha}$). If $\alpha\neq\beta$, then $A_{\alpha}\cap A_{\beta}=\emptyset$. The reason is: assume $\alpha<\beta$, if $x\in A_{\alpha}$, then $x\in X_{\alpha}$; if $x\in A_{\beta}$, since $\alpha<\beta$, so $x\in\kappa{\setminus}X_{\alpha}$. So $A_{\alpha}\cap A_{\beta}=\emptyset$.
\end{proof}

The following propositions are about \emph{Rudin-Keisler ordering} on ultrafilters. $U$ is an ultrafilter on $S$, $f\colon S\rightarrow T$, define $f_{*}(U)\subset P(T)$: $X\in f_{*}(U)$ iff $f^{-1''}X\in U$. Then $f_{*}(U)$ is an ultrafilter on $T$. Ultrafilters $U_1$ and $U_2$ are \emph{isomorphic} means there are $A_1\in U_1$, $A_2\in U_2$ and $h\colon A_1\rightarrow A_2$ such that $h$ is bijection and for $X\subset A_1$, $X\in U_1$ iff $h[X]\in U_2$. We use $U_1\equiv U_2$ to denote this property. ``$U_1$ and $U_2$ are isomorphic" is equivalent to ``$U_1\leq_{RK}U_2$ and $U_2\leq_{RK}U_1$". Obviously, many combinatorial properties share by ultrafilters which are isomorphic:

\begin{proposition}\label{tonggoubaochi}
Assume $U$ and $U'$ are two ultrafilters and $U\equiv U'$. $\kappa$ is an infinite cardinal. If $U$ is $\kappa$-complete, then $U'$ is also $\kappa$-complete. If $U$ has the property ``if $\langle T_{\beta}:\beta<\kappa\rangle$ is a sequence of elements of $U$, then there is $T\in U$ such that $\forall\beta<\kappa$, $|T{\setminus}T_{\beta}|<\kappa$", then $U'$ also has this property.
\end{proposition}

\begin{proposition}\label{lianglei}
$\kappa$ is an infinite cardinal, $U$ is a normal measure on $\kappa$, $f$ is a function with domain $\kappa$. Then (a) or (b) happens:
\begin{itemize}
\item[(a)]There is a constant function $g$, $g\sim_{U}f$.
\item[(b)]There is a one-to-one function $g$, $g\sim_{U}f$.
\end{itemize}
Consequently,
\begin{itemize}
\item[(1)] If $U'$ is a ultrafilter such that $U'\leq_{RK}U$, then $U'$ is principal or $U'\equiv U$.
\item[(2)] $U_1$ and $U_2$ are normal measures $\lambda_1$ and $\lambda_2$ respectively, $U_1\neq U_2$, then $U_1$ is not isomorphic to $U_2$.
\end{itemize}
\end{proposition}

\begin{proof}
This proposition is a corollary of Rowbottom theorem (see \cite{1} Theorem 10.22).

(1) $S\in U$, $T\in U'$, $h\colon S\rightarrow T$ and for $A\subset T$, $A\in U'$ iff $h^{-1}[A]\in U$. By normality of $U$, $h$ is $U$-equivalent to a constant function or a one-to-one function. If $h$ is $U$-equivalent to a constant function, then $U'$ is principal, so $h$ is $U$-equivalent to a one-to-one function. So $U'\equiv U$.

(2) If $\lambda_1\neq\lambda_2$, $A_1\in U_1$ and $A_2\in U_2$, then $|A_1|=\lambda_1$ and $|A_2|=\lambda_2$. So there is no bijection between them, so $U_1$ is not isomorphic to $U_2$. If $\lambda_1=\lambda_2$, $A_1\in U_1$ and $A_2\in U_2$, $h\colon A_1\rightarrow A_2$ is the isomorphism. $\{x\in A_1\colon h(x)<x\}\notin U_1$, the reason is if this set in $U_1$, then by normality, $h$ is $U_1$-equivalent to a constant function. Similarly, $\{x\in A_2\colon h^{-1}(x)<x\}\notin U_2$, so $\{x\in A_1\colon h(x)>x\}\notin U_1$. Thus $\{x\in A_1\colon h(x)=x\}\in U_1$, so $U_1=U_2$, this is a contradiction.
\end{proof}

\begin{example}
Suppose $U$ is an ultrafilter on $I$, $\langle U_i\colon i\in I\rangle$ is a family of ultrafilters on $X$. We can define another ultrafilter $\mu$ on $X$ as follows: $A\in\mu$ iff $\{i\in I\colon A\in U_i\}\in U$. Suppose the family $\langle U_i\colon i\in I\rangle$ has the following property: there is $\langle B_i\colon i\in I\rangle$ such that for each $i$, $B_i\in U_i$ and if $i\neq j$, then $B_i\cap B_j=\emptyset$. We still use $U_i$ to denote the ultrafilter $U_i{\uh}B_i$. Then $\mu$ and $U*\langle U_i\colon i\in I\rangle$ are isomorphic.
\end{example}

\begin{notation}\label{notation1}
\begin{itemize}
\item[(1)] $g$ is a function with domain $\bigoplus_{i\in I}B_i$, $i\in I$, let $g_i$ is a function on $B_i$ such that $g_i(x)=g(i,x)$. So given a function on $\bigoplus_{i\in I}B_i$ is equivalent to give a function sequence $\langle g_i\colon i\in I\rangle$ such that for each $i$, $\dom(g_i)=B_i$.
\item[(2)] Given a function $f$ with domain $I$, define $f'$, a function with domain $\bigoplus_{i\in I}B_i$ by for $i\in I$, $x\in B_i$, $f'(i,x)=f(i)$. We use $\mathcal{T}$ to denote the following class of functions: \[\mathcal{T}=\{f'\colon f\mbox{ is a function with domain }I\}.\]
\end{itemize}
\end{notation}

\begin{lemma}
If $I$ is an infinite cardinal and $U$ is a normal measure on $I$. $g\in\mathcal{T}$. Then $g$ is $U*\langle U_i\colon i\in I\rangle$-equivalent to a function with domain $\bigoplus_{i\in I}B_i$ of the following types:
\begin{itemize}
\item A constant function.
\item A function $h$ on $\bigoplus_{i\in I}B_i$ such that for each $i\in I$, $h_i$ is a one-to-one function on $B_i$.
\end{itemize}
\end{lemma}
\begin{proof}
This follows from normality of $U$ and proposition \ref{lianglei}.
\end{proof}

\begin{theorem}\label{hanshufenlei1}
Suppose $\kappa$ is an infinite cardinal, $|I|\leq\kappa$ and for each $i\in I$, $B_i=\kappa$ and $U_i$ is a normal measure on $\kappa$, and there is $K\in U$, such that if $i,j\in K$ and $i\neq j$, then $U_i\neq U_j$. If $g$ is a function with domain $I\times\kappa$ and $g$ is not equivalent to a function in $\mathcal{T}$, then $g$ is $U*\langle U_i\colon i\in I\rangle$-equivalent to a one-to-one function.
\end{theorem}

\begin{proof}
Given $g$. By normality of $U_i$'s, $g$ is equivalent to $\tilde{g}$, who satisfies each $\tilde{g}_i$ is a one-to-one function on $\kappa$.
Let $Y$ is a set such that $\tilde{g}\colon I\times\kappa\rightarrow Y$ and $|Y|=\kappa$. For each $i\in I$, $P_i=\tilde{g}_{*}(U_i)$ is an ultrafilter on $Y$ which is isomorphic to $U_i$ (the reason is $\tilde{g}_i$ is one-to-one). We have
\begin{itemize}
\item If $i\neq j$, then $P_i\neq P_j$. (By proposition \ref{lianglei})
\item For $i\in I$, $P_i$ is $\kappa$-complete. (By proposition \ref{tonggoubaochi})
\item If $\langle T_{\beta}\colon\beta<\kappa\rangle$ is a sequence of elements of $P_i$, then there is $T\in P_i$ such that for each $\beta<\kappa$, $|T{\setminus}T_{\beta}|<\kappa$. (By proposition \ref{tonggoubaochi})
\end{itemize}
By proposition \ref{bujiao}, there is $\langle A_i\colon i\in I\rangle$ such that for each $i\in I$, $A_i\in U_i$ and if $i\neq j$, then $A_i\cap A_j=\emptyset$. So $\tilde{g}{\uh}\bigoplus_{i\in I}A_i$ is a one-to-one function, but $\bigoplus_{i\in I}A_i\in U*\langle U_i\colon i\in I\rangle$. So $g$ is $U*\langle U_i\colon i\in I\rangle$-equivalent to a one-to-one function.
\end{proof}

\begin{theorem}[Principle of classification of functions]\label{hanshufenlei2}
$\langle\lambda_n\colon n<\omega\rangle$ is an increasing sequence of infinite cardinals( not necessary strictly). $\langle\pi_n\colon n<\omega\rangle$ and $\langle D_n\colon 0<n<\omega\rangle$ satisfies:
\begin{itemize}
\item[(1)] $\pi_0=D_1$ is a normal measure on $\lambda_0$.
\item[(2)] For $0<i<\omega$, $D_i$ is an ultrafilter on $\lambda_0\times...\times\lambda_{i-1}$. $\pi_i$ is a function with domain $\lambda_0\times...\times\lambda_{i-1}$ such that there is $K\in D_i$, $\pi_i{\uh}K$ is an one-to-one function and its values are all normal measures on $\lambda_{i-1}$.
\item[(3)] For $i<\omega$, $D_{i+1}=D_i*\pi_i$.
\end{itemize}
Then for a function $f$ with domain $\lambda_0\times...\times\lambda_{i-1}$, $f$ is $D_n$-equivalent to a function of the following types:
\begin{itemize}
\item[(0)] A constant function.
\item[(1)] A function $g$ defined from a function with domain $\lambda_0$, i.e. there is $h$ with domain $\lambda_0$, $g(x_0,...,x_{i-1})=h(x_0)$ for all $x_0,...,x_{i-1}$.
\item[...]
\item[(k)] A function $g$ defined from a function with domain $\lambda_0\times...\times\lambda_{k-1}$, i.e. there is $h$ with domain $\lambda_0\times...\times\lambda_{k-1}$, $g(x_0,...,x_{i-1})=h(x_0,...,x_{k-1})$ for all $x_0,...,x_{i-1}$.
\item[...]
\item[(i)] An one-to-one function with domain $\lambda_0\times...\times\lambda_{i-1}$.
\end{itemize}
\end{theorem}

\begin{proof}
By induction, this is a corollary of theorem \ref{hanshufenlei1}.
\end{proof}

In fact, sum of ultrafilters are just an ultrafilter yields two steps ultrapower: if $U$ is a countably complete ultrafilter on $I$, then $N=\Ult(V,U)$ and induced elementary embedding $j$ are well-defined. $B$ is a set, $U'$ is an ultrafilter on $j(B)$ in $N$. Pick a represent function $e$ of $U'$ such that for each $i\in I$, $f(i)$ is an ultrafilter on $B$. Then we define $U*U'=U*\langle f(i)\colon i\in I\rangle$. This definition is well: $U*U'$ is an ultrafilter on $I\times B$ and it is independent of the choice of the represent function $f$. This illustrates that this kind of product ultrafilter is still a particular case of sum of ultrafilters.

\begin{corollary}
Suppose $\lambda$ is an infinite cardinal. If $U$ is $\lambda$-complete and $U'$ is $j(\lambda)$-complete, then $U*U'$ is $\lambda$-complete.
\end{corollary}

\begin{proof}
This is an immediate corollary of proposition \ref{lambdacomplet}.
\end{proof}

So if $U'$ is countably complete in $N$, then $U*U'$ is countably complete, thus the ultrapower $\Ult(V,U*U')$ and induced embedding $j'$ is well-founded.
In this case, by definition of $U*U'$, we have the following proposition. It tells us that the ultrapower by $U*U'$ is just ultrapower by $U$ and then ultrapower by $U'$, i.e. $\Ult(V,U*U')$ is a two step ultrapower.

\begin{proposition}
\begin{itemize}
\item[(1)] Let $\mathcal{T}=\{f'\colon f\mbox{ is a function with domain }I\}$ (as in Notation \ref{notation1}), then $\Ult(\mathcal{T},U*U')=N$ and the induced embedding is equal to $j$.
\item[(2)] $N_2=\Ult(N,U')$ and $j_2$ is the induced embedding. Then $\Ult(V,U*U')=N_2$ and $j'=j_2\circ j$. Then canonical embedding (the identity map from $\mathcal{T}$ to $V$) from $\Ult(\mathcal{T},U*U')=N$ to $\Ult(V,U*U')=N_2$ is equal to $j_2$.
\end{itemize}
\end{proposition}

\section{Forcing notion $\mathbb{P}_{\vec{\mu}}$}

\subsection{Definition and properties}

\begin{definition}
$\kappa$ is an infinite cardinal, $\vec{\mu}=\langle\mu_{\alpha}:\alpha<\kappa\rangle$ is a sequence of distinct $\kappa$-complete ultrafilters on $\kappa$. Define $\mathbb{P}_{\vec{\mu}}$ is the forcing notion as follows:
\begin{itemize}
\item Forcing conditions are pairs $(s,F)$, such that $s\in[\kappa]^{<\omega}{\setminus}\{\emptyset\}$, and $F$ is a function with domain $\kappa$ and for each $\alpha<\kappa$, $F(\alpha)\in\mu_{\alpha}$. Such a function $F$ is called \emph{$\vec{\mu}$-choice function}.
\item $(s,F)$ and $(s',F')$ are conditions, $(s',F')$ is stronger than $(s,F)$ iff
\begin{itemize}
\item[(a)] $s\subset s'$ and $s'\cap(\max(s)+1)=s$.
\item[(b)] For $\alpha<\kappa$, $F'(\alpha)\subset F(\alpha)$.
\item[(c)] $\forall i(|s|\leq i<|s'|\rightarrow s'(i)\in F(s'(i-1)))$.
\end{itemize}
\end{itemize}
\end{definition}

\begin{remark}
One can obtain a measurable cardinal $\kappa$ which carries $\kappa$ many (in fact, $\kappa^{++}$ many) normal measures by forcing (see \cite{2}).
\end{remark}

Suppose $G$ is a $\mathbb{P}_{\vec{\mu}}$-generic filter over $V$, then define $g\colon\omega\rightarrow\kappa$ by for each $0<n<\omega$, there is $F$ such that $(\{g(0),...,g(n-1)\},F)\in G$. We say $g$ is the \emph{generic sequence} corresponding to $G$. $g$ is cofinal in $\kappa$, so in $V[G]$, $\cf(\kappa)=\omega$.
Also, we have
\[G=\{(s,F)\in\mathbb{P}_{\vec{\mu}}:g{\uh}|s|=s\ \&\ \forall i(|s|\leq i\rightarrow g(i)\in F(g(i-1)))\}.\]
So $V[G]=V[g]$.

For $(s,F)\in\mathbb{P}_{\vec{\mu}}$, $|s|=n$ and $\tilde{s}=\langle s(0),...,s(n-1)\rangle$, define $T(s,F)$ as follows:
\[T(s,F)=\{t\in\kappa^{<\omega}{\setminus}\{\emptyset\}\colon (t=s{\uh}|t|)\ \vee\ (t{\uh}n=s\ \&\ \forall i(n\leq i<|t|\rightarrow t(i)\in F(t(i-1))))\}.\]

Then $T(s,F)$ is a tree under the ordering ``extension" of sequences. Define $\mathbb{P}_{\vec{\mu}}^{'}$ is the forcing notion with $T(s,F)$, for $(s,F)\in\mathbb{P}_{\vec{\mu}}$, are conditions, and $T(s',F')$ is stronger than $T(s,F)$ in $\mathbb{P}_{\vec{\mu}}^{'}$ iff $T(s',F')$ is a subtree of $T(s,F)$.
So we have defined a mapping $T\colon\mathbb{P}_{\vec{\mu}}\rightarrow\mathbb{P}_{\vec{\mu}}^{'}$, sending $(s,F)$ to $T(s,F)$. Clearly, $T$ is onto. Moreover,

\begin{proposition}
For $p,q\in\mathbb{P}_{\vec{\mu}}$,
\begin{itemize}
\item[(1)] $T(p)=T(q)$ iff $p\sim q$.
\item[(2)] $T(p)$ is a subtree of $T(q)$ iff $p\leq^{*}q$.
\end{itemize}
So $T\colon\mathbb{P}_{\vec{\mu}}\rightarrow\mathbb{P}_{\vec{\mu}}^{'}$ is isomorphic to the separative quotient mapping of $\mathbb{P}_{\vec{\mu}}$. So $\mathbb{P}_{\vec{\mu}}$ and $\mathbb{P}_{\vec{\mu}}^{'}$ are equivalent forcing notions.
\end{proposition}

Suppose $p,q$ are two conditions in a forcing notion $P$, $p\sim q$ if for all $r\in P$, $r$ is compatible with $p$ iff $r$ is compatible with $q$. $p\sim q$ iff for $G$, a $P$-generic filter over $V$, $p,q\in G$ or $\{p,q\}\cap G=\emptyset$. $p\leq^{*}q$ if for all $r\in P$, if $r\leq p$, then $r$ is compatible with $q$. $p\leq^{*}q$ iff $p\Vdash q\in G$. $(P/\sim,\leq^{*})$ is called the separative quotient of $P$, the induced onto mapping from the equivalence relation $\sim$ is called separative quotient mapping of $P$. The separative quotient of $P$ is \emph{separable}(for all $p,q\in G$, if $p\Vdash q\in G$, then $p\leq q$), and equivalent to $P$. See \cite{21} for these concepts about forcing.

\begin{notation}\label{notationfortree}
For $p,q\in\mathbb{P}_{\vec{\mu}}^{'}$,
\begin{itemize}
\item[(1)] $\stem(p)$ is the largest element of $p$ which is comparable with all elements of $p$. We also identify $\stem(p)$ with the finite set $\ran(\stem(p))$.
\item[(2)] For $n\in\omega$, $p(n)$ is the subset of $p$ consists of elements of height $|\stem(p)|+n$.
\item[(3)] For $x\in p$ and $\stem(p)\leq x$, let $p^{(x)}$ be the subtree of $p$ consists of all elements of $p$ which is comparable with $x$. Thus $\stem(p^{(x)})=x$.
\item[(4)] We say $q$ is a \emph{pure extension} of $p$ if $q$ is stronger than $p$ and $p,q$ has the same stem.
\end{itemize}
\end{notation}

Given $\alpha<\kappa$, let us define a sequence $\langle E_n^{\alpha}\colon n<\omega\rangle$ as follows:
\begin{itemize}
\item $E_n^{\alpha}$ is an ultrafilter on $\kappa^n$, $E_1^{\alpha}=\mu_{\alpha}$;
\item $E_{n+1}^{\alpha}=E_n^{\alpha}*\vec{\mu}$. (The definition of $*$, see Definition \ref{def:sum})
\end{itemize}

If $p\in\mathbb{P}_{\vec{\mu}}^{'}$, then $E_n^p$ is just $E_n^{\alpha}$ where $\alpha$ is the last element of $\stem(p)$. $E_n^p$ is just related with $\stem(p)$. If we have fixed $\alpha$ or $p$, then $E_n$ has no confusion.
By proposition \ref{lambdacomplet}, each $E_n^{\alpha}$ is $\kappa$-complete. Moreover, we have:

\begin{proposition}
$p\in\mathbb{P}_{\vec{\mu}}^{'}$. Then
\begin{itemize}
\item[(1)] For $n>0$, $p(n)\in E_n^{p}$.
\item[(2)] $q$ is a subtree of $p$ and $\stem(q)=\stem(p)$, then $q\in\mathbb{P}_{\vec{\mu}}^{'}$ iff for $n>0$, $q(n)\in E_n^p$.
\item[(3)] If $A\subset p(n)$, define $p^A=\bigcup_{x\in A}p^{(x)}$. Then $A\in E_n$ iff $p^{A}\in\mathbb{P}_{\vec{\mu}}^{'}$.
\item[(4)] $\langle p_n\colon n<\omega\rangle$ is a sequence of elements of $\mathbb{P}_{\vec{\mu}}^{'}$, and have the same stem, then $\bigcap_{n<\omega}p_n\in\mathbb{P}_{\vec{\mu}}^{'}$.
\end{itemize}
\end{proposition}

This proposition can be used to shrink a tree to an appropriate extent, strong enough and it is also an element of $\mathbb{P}_{\vec{\mu}}^{'}$. The sequence also has such a useful property:

\begin{corollary}
For a function $f$ with domain $\kappa^i$, $f$ is $E_n^{\alpha}$-equivalent to a function of the following types:
\begin{itemize}
\item[(0)] A constant function.
\item[(1)] A function $g$ defined from a function with domain $\kappa$, i.e. there is $h$ with domain $\kappa$, $g(x_0,...,x_{i-1})=h(x_0)$ for all $x_0,...,x_{i-1}$.
\item[...]
\item[(k)] A function $g$ defined from a function with domain $\kappa^{k}$, i.e. there is $h$ with domain $\kappa^{k}$, $g(x_0,...,x_{i-1})=h(x_0,...,x_{k-1})$ for all $x_0,...,x_{i-1}$.
\item[...]
\item[(i)] A one-to-one function with domain $\kappa^i$.
\end{itemize}
\end{corollary}

\begin{proof}
This is an immediate corollary of Theorem \ref{hanshufenlei2}.
\end{proof}

\begin{lemma}\label{geo0}
Suppose $D$ is a dense subset of $\mathbb{P}_{\vec{\mu}}^{'}$, $p\in\mathbb{P}_{\vec{\mu}}^{'}$. Then there is $\tilde{p}$, which is a pure extension of $p$, and $n\in\omega$, such that if $x\in \tilde{p}(n)$, then $\tilde{p}^{(x)}\in D$.
\end{lemma}

\begin{proof}
Given $p\in\mathbb{P}_{\vec{\mu}}^{'}$. We claim that there is $0<n<\omega$, \[\{x\in p(n)\colon p^{(x)}\mbox{ has a pure extension in }D\}\in E_n.\]
Suppose not. For $0<n<\omega$, $\{x\in p(n)\colon p^{(x)}\mbox{ has no pure extension in }D\}\in E_n$. Then there is $q$, a pure extension of $p$, such that for all $x\in q$, $q^{(x)}$ has no pure extension in $D$. But since $D$ is dense, so there is $r\leq q$, $r\in D$. $r$ is a pure extension of $q^{(x)}$ for some $x\in q$. A contradiction, this proved the claim.

So there is $\tilde{p}$, a pure extension of $p$, such that if $x\in \tilde{p}(n)$, then $\tilde{p}^{(x)}\in D$.
\end{proof}

The following variant of lemma \ref{geo0} is also useful, the proof from lemma \ref{geo0} is clear:
\begin{lemma}\label{geo1}
Suppose $D$ is a dense subset of $\mathbb{P}_{\vec{\mu}}^{'}$, $s\in[\kappa]^{<\omega}{\setminus}\{\emptyset\}$, then there is a $\vec{\mu}$-choice function $F$ and $n\in\omega$ such that if $t\in T(s,F)$ and $|t|=|s|+n$, then $(t,F)\in T^{-1}[D]$.
\end{lemma}

\begin{theorem}[Prikry property]
Suppose $\varphi$ is a sentence in the language of the forcing $\mathbb{P}_{\vec{\mu}}^{'}$, $p\in\mathbb{P}_{\vec{\mu}}^{'}$, then there is $q\in\mathbb{P}_{\vec{\mu}}^{'}$, a pure extension of $p$, such that $q$ decides $\varphi$, i.e. $q\Vdash\varphi$ or $q\Vdash\neg\varphi$.
\end{theorem}

\begin{proof}
$D=\{p\colon p\mbox{ decides }\varphi\}$ is dense in $\mathbb{P}_{\vec{\mu}}^{'}$. So by the lemma above, there is $n<\omega$ and $\tilde{p}$, a pure extension of $p$, such that if $x\in \tilde{p}(n)$, then $\tilde{p}^{(x)}\in D$. So $A_1=\{x\colon \tilde{p}^{(x)}\Vdash\varphi\}$ or $A_2=\{x\colon \tilde{p}^{(x)}\Vdash\neg\varphi\}$ is in $E_n$. So there is a pure extension $q$, such that $q(n)=A_i$ when $A_i\in E_n$. $q$ decides $\varphi$.
\end{proof}

Similarly, we have

\begin{theorem}[Prikry property]
Suppose $\varphi$ is a sentence in the language of the forcing $\mathbb{P}_{\vec{\mu}}$, $s\in[\kappa]^{<\omega}{\setminus}\{\emptyset\}$. Then there is $F$, a $\vec{\mu}$-choice function, $(s,F)$ decides $\varphi$.
\end{theorem}

\begin{proposition}
\begin{itemize}
\item[(1)] $\mathbb{P}_{\vec{\mu}}$ does not add new bounded subset of $\kappa$.
\item[(2)] $\mathbb{P}_{\vec{\mu}}$ has $\kappa^{+}$-c.c.
\item[(3)] $\mathbb{P}_{\vec{\mu}}$ preserves all cardinals and preserves all cofinalities except $\kappa$.
\end{itemize}
\end{proposition}

\begin{proof}
\noindent(1) Suppose $G$ is a $\mathbb{P}_{\vec{\mu}}$-generic filter over $V$. In $V[G]$, $\theta<\kappa$ and $A\subset\theta$. $\dot{A}$ is a name for $A$, $\Vdash\dot{A}\subset\theta$. Given $s$, for $\alpha<\theta$, $F_{\alpha}$ is a $\vec{\mu}$-choice function such that $(s,F_{\alpha})$ decides the sentence ``$\alpha\in\dot{A}$". Define $H$: for $x\in\kappa$, $H(x)\bigcap_{\alpha<\theta}F_{\alpha}(x)$. Since every ultrafilter in $\vec{\mu}$ is $\kappa$-complete, $H$ is also a $\vec{\mu}$-choice function. And for each $\alpha<\theta$, $(s,H)$ decides the sentence ``$\alpha\in\dot{A}$". So there is such a $(s,H)\in G$. Thus $A=\{\alpha<\theta\colon(s,H)\Vdash\alpha\in\dot{A}\}\in V$.

\noindent(2) If $(s,F),(s',F')\in\mathbb{P}_{\vec{\mu}}$ are incompatible, then $s\neq s'$. So an antichain is of size less than or equal to $|[\kappa]^{<\omega}|=\kappa$, so $\mathbb{P}_{\vec{\mu}}$ has $\kappa^{+}$-c.c.

\noindent(3) From (1) and (2), it is clear.
\end{proof}

\begin{theorem}[Geometric condition]\label{geometric}
$M$ is a transitive model of $\ZFC$, in $M$, $\kappa$ is an infinite cardinal and $\vec{\mu}=\langle\mu_{\alpha}\colon\alpha<\kappa\rangle$ is a sequence of distinct normal measures on $\kappa$, $S\colon\omega\rightarrow\kappa$. Then $S$ is a $(\mathbb{P}_{\vec{\mu}})^{M}$-generic sequence over $M$ iff for each $\vec{\mu}$-choice function $F$ in $M$, there is $0<n<\omega$, such that $\forall i(n\leq i<\omega\rightarrow S(i)\in F(S(i-1)))$.
\end{theorem}

In the proof of this theorem, the following lemma are used:

\begin{lemma}\label{geo2}
$\langle F_s\colon s\in[\kappa]^{<\omega}{\setminus}\{\emptyset\}\rangle$ is a family of $\vec{\mu}$-choice function, then there is a $\vec{\mu}$-choice function $H$ such that for all $s$, $(s,H)\leq^{*}(s,F_s)$.
\end{lemma}

\begin{proof}
For $\alpha<\kappa$, define $H_{\alpha}=\bigcap_{\max(s)\leq\alpha}F_s(\alpha)$. $s=\{s(0),...,s(n-1)\}$, $w=\langle s(0),...,s(n-1),t(0),...,t(m-1)\rangle\in T(s,H)$. For $x=s(n-1),t(0),...,t(m-2)$, $H(x)=\bigcap_{\max(t)\leq x}F_t(x)\subset F_s(x)$. Because $w\in T(s,H)$, so $w\in T(s,F_s)$.
\end{proof}

\begin{proof}[Proof of Theorem \ref{geometric}]

\noindent``$\Longrightarrow$" $G$ is the corresponding generic filter to $S$. Then $S$ is a branch of $T(s,F)$ for $(s,F)\in G$. So $F$ and $|s|$ are what we want.

\noindent``$\Longleftarrow$" Suppose $D$ is a dense subset of $(\mathbb{P}_{\vec{\mu}}^{'})^{M}$. To prove there is $s$, an initial segment of $S$ and a $\vec{\mu}$-choice function $F$ such that $(s,F)\in T^{-1}[D]$ and $S$ is a branch of $T(s,F)$. For each $s$, let $F_s$ and $n_s$ as described in lemma \ref{geo1}. Let $H$ be the $\vec{\mu}$-choice function described in lemma \ref{geo2}.

There is $m$, for all $i$, if $m\leq i<\omega$, then $S(i)\in H(S(i-1))$. $w=\{S(0),...,S(m-1)\}$, then $S$ is a branch of $T(w,H)$. Then $T(w,H)$ is a subtree of $T(w,F_w)$. So $S$ is also a branch of $T(w,F_w)$.
Let $s=S{\uh}(m+n_w)$. Then $(s,F_w)\in T^{-1}[D]$.
\end{proof}

\begin{corollary}
$M$ is a transitive model of $\ZFC$, $\kappa,\vec{\mu}\in M$,
\begin{itemize}
\item[(1)] Suppose $g$ is a generic sequences of $(\mathbb{P}_{\vec{\mu}})^{M}$ over $M$ and $a\colon\omega\rightarrow\kappa$ and $\{n\in\omega:g(n)\neq a(n)\}$ is finite, then $a$ is also a generic sequences of $(\mathbb{P}_{\vec{\mu}})^{M}$ over $M$.
\item[(2)] Suppose $g_1$ and $g_2$ are two generic sequences of $(\mathbb{P}_{\vec{\mu}})^{M}$ over $M$. Then $A=\{n\in\omega\colon g_1(n)=g_2(n)\}$ is finite or $\omega{\setminus}A$ is finite.
\end{itemize}
\end{corollary}

\begin{proof}
(1)By geometric condition.

(2)$F$ is a $\vec{\mu}$-choice function in $M$ such that if $\alpha\neq\beta$, then $F(\alpha)\cap F(\beta)=\emptyset$.
$N\in\omega$ is sufficent large such that for all $i>N$, $g_{1}(i)\in F(g_{1}(i-1))$ and $g_{2}(i)\in F(g_{2}(i-1))$.
If there is $n>N$, $g_{1}(n)\neq g_{2}(n)$, then by the assumption on $F$, $g_{1}(n+1)\neq g_{2}(n+1)$; so $g_{1}(n+2)\neq g_{2}(n+2)$,...So for $i>N$, $g_{1}(i)\neq g_{2}(i)$. In this case, $A$ is finite. Otherwise, for any $n>N$, $g_{1}(n)=g_{2}(n)$, in this case, $\omega{\setminus}A$ is finite.
\end{proof}

\subsection{$\mathbb{P}_{\vec{\mu}}$ yields minimal extensions}
In this subsection, all elements of $\vec{\mu}$ are normal measures on $\kappa$. Fix a $\vec{\mu}$-choice function $K$ such that if $\alpha\neq\beta$, then $K(\alpha)\cap K(\beta)=\emptyset$, the existence of such $K$ depends on proposition \ref{bujiao} and normality of each $\mu_{\alpha}$. Define $W=\{(s,F)\in\mathbb{P}_{\vec{\mu}}\colon\forall\alpha<\kappa\ F(\alpha)\subset K(\alpha)\}$, $W$ is a dense open subset of $\mathbb{P}_{\vec{\mu}}$. $W'=\{T(s,F)\colon(s,F)\in P\}$ is a dense subset of $\mathbb{P}_{\vec{\mu}}^{'}$. We use $\mathbb{P}_{\vec{\mu}}^{''}$ to denote the forcing notion $W'$ with the subtree ordering. Then $\mathbb{P}_{\vec{\mu}}$, $\mathbb{P}_{\vec{\mu}}^{'}$ and $\mathbb{P}_{\vec{\mu}}^{''}$ are equivalent. Given $p\in\mathbb{P}_{\vec{\mu}}^{''}$, by the property of $K$, if $x,y\in p$, then $\max(x)\neq\max(y)$, so we can identify $p$ with a subset of $\kappa$, i.e. using $\max(x)$ to replace $x$ in $p$.

\begin{notation}
Assume $\dot{a}$ is a $\mathbb{P}_{\vec{\mu}}$-name such that $\Vdash\dot{a}\colon\kappa\rightarrow 2$. $p\in\mathbb{P}_{\vec{\mu}}^{''}$, define $a_p\colon\kappa\rightarrow 2$ as follows: for $i\in\{0,1\}$, $a_p(\alpha)=i$ iff there is a pure extension $q$ of $p$, $q\Vdash\dot{a}(\alpha)=i$. By Prikry property, this definition is well. Note that $a_p$ is just related with $\stem(p)$, i.e. if $\stem(p)=\stem(q)$, then $a_p=a_q$. So the number of all $a_p$'s is $\kappa$, this because $|[\kappa]^{<\omega}|=\kappa$.
\end{notation}

There is a simple observation: $\theta$ is a limit ordinal. $x$ and $y$ are different functions on $\theta$ with value in $\{0,1\}$. Define $\delta(x,y)$ be the least $\alpha$ such that $x{\uh}\alpha\neq y{\uh}\alpha$. Then $\delta(x,y)$ is a successor ordinal less than $\theta$. If $a,b,c$ are functions on $\theta$ with value in $\{0,1\}$ such that $a\neq b$ and $a\neq c$. Then the followings are equivalent:

\begin{itemize}
\item[(1)] $\delta(a,b)=\delta(a,c)$;
\item[(2)] $b{\uh}\delta(a,b)$ and $c{\uh}\delta(a,c)$ are comparable;
\item[(3)] $b{\uh}\delta(a,b)=c{\uh}\delta(a,c)$.
\end{itemize}

Given $p\in\mathbb{P}_{\vec{\mu}}^{''}$, the sequence $\langle E_n^p\colon 0<n<\omega\rangle$ has such a property: classification of functions as in Corollary \ref{hanshufenlei2}.

\begin{lemma}\label{kappaziji}
Suppose $G$ is a $\mathbb{P}_{\vec{\mu}}$-generic filter over $V$. In $V[G]$, $A\subset\kappa$ and $A\notin V$. Then $V[A]=V[G]$.
\end{lemma}

\begin{proof}
Given $G$ and $A$. $\dot{a}$ is a $\mathbb{P}_{\vec{\mu}}$-name, $\Vdash\dot{a}\colon\kappa\rightarrow 2$, and $\Vdash\dot{a}\notin V$. $\dot{a}$ is a name for character function of $A$, i.e. $A=\{\alpha<\kappa\colon(\dot{a}/G)(\alpha)=1\}$.

For $p\in\mathbb{P}_{\vec{\mu}}^{''}$, $\Psi(p)$ is the abbreviation of the formula:
\begin{alignat*}{1}
\text{For }0<n<\omega,\ \{x\in p(n)\colon a_{p^{(x)}}=a_p\}\in E_{n}^{p}.
\end{alignat*}
\begin{claim}
If $\Psi(p)$ holds, then there exists $q$, a pure extension of $p$ such that for each $x\in q$, $a_{q^{(x)}}=a_q$. This $q$ forces ``$\dot{a}=a_q$", so $q\Vdash\dot{a}\in V$.
\end{claim}
The first statement is by $\kappa$-completeness of each $E_{n}^{p}$, and $a_p=a_q$ (since $q$ is a pure extension). To prove the second statement: suppose not. $q\nVdash\dot{a}=a_q$. Then there is $r$, a subtree of $q$, $i\in\{0,1\}$ and $\alpha<\kappa$, $r\Vdash\dot{a}(\alpha)=i$ and $a_q(\alpha)\neq i$. $t=\stem(r)$, because $r\leq q$, so $r\leq q^{(t)}$. $r$ is a pure extension of $q^{(t)}$, so $a_r(\alpha)=a_q(\alpha)\neq i$; but $r\Vdash\dot{a}=i$, so $a_r(\alpha)=i$, this is a contradiction. So we proved the Claim 1. Claim 1 tells us, for each $p\in\mathbb{P}_{\vec{\mu}}^{''}$, $\neg\Psi(p)$, i.e. there is $n>0$, $\{x\in p(n)\colon a_{p^{(x)}}\neq a_p\}\in E_{n}^{p}$.

For $q$ and $\mathcal{C}$, $\Omega(q,\mathcal{C})$ is the conjunction of the followings:
\begin{itemize}
\item[(1)] $q\in\mathbb{P}_{\vec{\mu}}^{''}$. $\mathcal{C}$ is a function such that $\stem(q)\in\dom(\mathcal{C})\subset q$.
\item[(2)] For $x\in\dom(\mathcal{C})$, $\mathcal{C}(x)=(m_x,f_x)$, $0<m_x<\omega$ and $f_x$ is a function with domain $q^{(x)}(m_x)$. For $t\in\dom(f_x)$, $f_x(t)=(\delta_x(t),\sigma_x(t))$, $\delta_x(t)<\kappa$ and $\sigma_x(t)$ is a $0$-$1$ sequence of length $\delta_x(t)$. For $s,t\in\dom(f_x)$ and $s\neq t$, $\delta_x(t)\neq\delta_x(s)$ and $\sigma_x(t)$ is incomparable with $\sigma_x(s)$.
\item[(3)] For $x\in\dom(\mathcal{C})$, if $t\in q^{(x)}(m_x)$, then $q^{(t)}$ forces $\dot{a}{\uh}\delta_x(t)=\sigma_x(t)$.
\item[(4)] If $x\in\dom(\mathcal{C})$, then for all $t\in q(m_x)$, $t\in\dom(\mathcal{C})$. If $x\in\dom(\mathcal{C})$, $x=\stem(q)$ or there is $s\leq_q x$ such that $s\in\dom(\mathcal{C})$ and $x\in q^{(s)}(m_s)$.
\end{itemize}

\begin{claim}
$\forall p\exists q(q\mbox{ is a pure extension of }p\mbox{ and }\exists\mathcal{C}\ \Omega(q,\mathcal{C}))$.
\end{claim}

Let us define of the operation on $p$: there is $n>0$, such that $B=\{x\in p(n)\colon a_{p^{(x)}}\neq a_p\}\in E_{n}^{p}$ by Claim 1. Then we define $p'=p^{B}$, $p'$ is a subtree of $p$ and $p'\in\mathbb{P}_{\vec{\mu}}^{''}$. For $x\in p'(n)$, let $\delta(x)$ be the least ordinal $\gamma$ such that $a{\upharpoonright}\gamma\neq a_{{p'}^{(x)}}{\upharpoonright}\gamma$. So $\delta$ is a function with domain $p'(n)$. By the property of the sequence $\langle E_n^p\colon 0<n<\omega\rangle$, there is $m\leq n$, a subtree $p''$, and a one-to-one function $f$ on $p''(m)$ such that for $t\in p''(m)$, $x,y\in p''(n)$ and $t<_p x$ and $t<_p y$, $f(x)=f(y)$. So we can understand that $\delta$ is a function on $p''(m)$. Define a subtree $p'''$ of $p''$ such that for $x\in p'''(m)$, ${p'''}^{(x)}$ decides the value of $\dot{a}{\uh}\delta(x)$. Let us use $b_x$ to denote this value, i.e. $b$ is a $0$-$1$ sequence of length less than $\kappa$. Then if $x\neq y$ and $x,y\in p'''(m)$, then $b_x$ and $b_y$ are incomparable, this because the observation before this lemma. Then output $p'''$, $m$ and the $\delta$-function on $p'''(m)$.

We construct $q$ and $\mathcal{C}$ simultaneously. Given $p$, let us define an operator on $p$, the result of this operator is a subtree $p_1$ of $p$, a natural number $m$ and a function on $p_1(m)$. And then use this operator for $p_1^{(x)}$ where $x\in p_1(m)$, iterate this process $q$ and $\mathcal{C}$ are defined. We proved Claim 2.

Thus $\{q\in\mathbb{P}_{\vec{\mu}}^{''}\colon \exists\mathcal{C}\ \Omega(q,\mathcal{C})\}$ is a dense subset of $\mathbb{P}_{\vec{\mu}}^{''}$. So there is $q$ and $\mathcal{C}$ such that $q\in G$ and $\Omega(q,\mathcal{C})$.
We can use $q$, $\mathcal{C}$ and $A$ to compute $g$, the generic sequence corresponding to $G$. Let $h\colon\kappa\rightarrow 2$ be the character function of $A$. $\alpha=\stem(q)$. $\mathcal{C}(\alpha)=(m,f)$. Since for $x,y\in q(m)$ and $x\neq y$, $f(x)$ and $f(y)$ are incomparable, then there is exact one $f(x)$ which is comparable with $h$. Then this $x$ is on the path of $g$. From $q^{(x)}$, repeat this process. In $\omega$ steps, $g$ is defined.
\end{proof}

\begin{theorem}\label{Qminimal}
Forcing notion $\mathbb{P}_{\vec{\mu}}$ yields minimal extensions.
\end{theorem}

\begin{proof}
Suppose $G$ is a $\mathbb{P}_{\vec{\mu}}$-generic filter over $V$. In $V[G]$, $A\subset\Ord$ and $A\in V$. $\theta$ is the least ordinal such that $A\cap\theta\notin V$, i.e. $A\cap\theta\notin V$ and $\forall\alpha<\theta\ A\cap\alpha\in V$. Then $\theta$ is a limit ordinal and $A\cap\theta$ is unbounded in $\theta$. (The reason is: if $\theta=\beta+1$, $A\cap\theta$ is $A\cap\beta$ or $(A\cap\beta)\cup\{\beta\}$. Since $A\cap\theta\notin V$, so $A\cap\beta\notin V$, this contradicts the minimality of $\theta$.) Let $\delta=\cf(\theta)$, then $\delta$ is a regular cardinal. Fix $\langle\theta_{\alpha}:\alpha<\delta\rangle$ is a strictly increasing sequence of ordinals of supremum $\theta$.

$\dot{A}$ is a name for $A$ such that $\Vdash``\dot{A}\cap\theta\notin V,\ \forall\alpha<\theta (\dot{A}\cap\alpha\in V)"$.

\begin{claim}
$\delta\leq\kappa$.
\end{claim}
Suppose $\delta>\kappa$.
Let $X_{\alpha}=A\cap\theta_{\alpha}$, so $X_{\alpha}\in V$. For $\alpha<\delta$, there is $(s_{\alpha},F_{\alpha})\in G$, such that $(s_{\alpha},F_{\alpha})\Vdash\dot{A}\cap\theta_{\alpha}=X_{\alpha}$. So there is $s$ and a unbounded set $K\subset\delta$ such that $\forall\alpha\in K\ s_{\alpha}=s$. The reason is: the number of such $s$ is $\kappa$, if for each $s$, the occurrence of $s$ is bounded in $\theta$, so the number is less than $\theta$, but $\theta$ is regular, $\theta>\kappa$, this is a contradiction. Fix such a $s$. For $\alpha\in K$, there is $b\in V$ such that $(s,F_{\alpha})\Vdash\dot{A}\cap\theta_{\alpha}=b$. For $\alpha\in K$, $b_{\alpha}$ is a $b$ as above. Then for $\alpha,\beta\in K$ and $\alpha<\beta$, $b_{\alpha}=b_{\beta}\cap\theta_{\alpha}$. The reason is: there is $(s,H)\in G$, stronger than $(s,F_{\alpha})$ and $(s,F_{\beta})$, so $(s,H)$ forces $\dot{A}\cap\theta_{\alpha}=b_{\alpha}$ and $\dot{A}\cap\theta_{\beta}=b_{\beta}$, but $\theta_{\alpha}<\theta_{\beta}$, so $b_{\alpha}=b_{\beta}\cap\theta_{\alpha}$. Define $b=\bigcup_{\alpha\in K}b_{\alpha}$. Then $b\subset\theta$ and $b\in V$. For $\alpha\in K$, there is $p\in G$, $p\Vdash\dot{A}\cap\theta_{\alpha}=b_{\alpha}$, and $K$ is unbounded in $\delta$, so in $V[G]$, $A\cap\theta=b$. So $A\cap\theta\in V$. A contradiction. This proved Claim.

\noindent \underline{\emph{Case 1.}} $\theta\leq\kappa$.

$A\cap\theta\subset\kappa$, but $A\cap\theta\notin V$, so by lemma \ref{kappaziji}, $V[A\cap\theta]=V[G]$. Since $V[A\cap\theta]\subset V[A]$, so $V[A]=V[G]$.

\noindent \underline{\emph{Case 2.}} $\theta>\kappa$.
Define\[S=\{(\alpha,a):\alpha<\delta\mbox{ and }\exists p\ p\Vdash\dot{A}\cap\theta_{\alpha}=a\}.\]For each $\alpha<\theta$, if $(\alpha,a),\ (\alpha,b)\in S$ and $a\neq b$, then there is $p$ and $q$ such that $p\Vdash\dot{A}\cap\theta_{\alpha}=a$ and $q\Vdash\dot{A}\cap\theta_{\alpha}=b$, so $p$ is incompatible with $q$, but the forcing notion $\mathbb{P}_{\vec{\mu}}$ has $\kappa^{+}$-chain condition, so for $\alpha<\delta$, the number of $a$ such that $(\alpha,a)\in S$ is less than or equal to $\kappa$. Since $\delta\leq\kappa$, so $|S|\leq\kappa$.
Fix a one-to-one mapping $F\colon S\rightarrow\kappa$.

For $\alpha<\theta$, $X_{\alpha}=A\cap\delta_{\alpha}$. Then $\{(\alpha,X_{\alpha}):\alpha<\delta\}\subset S$. Let $B=F[\{(\alpha,X_{\alpha}):\alpha<\delta\}]$. $\{(\alpha,X_{\alpha}):\alpha<\delta\}$ is just $\langle X_{\alpha}:\alpha<\delta\rangle$. Then \[V[B]=V[\{(\alpha,X_{\alpha}):\alpha<\delta\}]=V[\langle X_{\alpha}:\alpha<\delta\rangle]=V[A\cap\theta].\] Since $A\cap\theta\notin V$, so $B\notin V$, but $B\subset\kappa$, by lemma \ref{kappaziji}, $V[A\cap\theta]=V[B]=V[G]$. so $V[A]=V[G]$.
\end{proof}

\section{Forcing notion $\mathbb{P}_{\mathcal{D}}$}

\subsection{Definition and properties}

In some sense, the development of the theory of the forcing $\mathbb{P}_{\mathcal{D}}$ is parallel to $\mathbb{P}_{\vec{\mu}}$. For this reason, we omit some proofs in detail, if the corresponding part for $\mathbb{P}_{\vec{\mu}}$ is exactly similar.

\begin{definition}
$\langle\gamma_n:n<\omega\rangle$ is a strictly increasing sequence of infinite cardinals, $\gamma$ is the supremum of $\langle\gamma_n:n<\omega\rangle$. $\mathcal{D}$ is a function with domain $\{(0,0)\}\cup\{(n,\alpha)\colon 0<n<\omega\ \&\ \alpha<\gamma_{n-1}\}$, such that
\begin{itemize}
\item $\mathcal{D}(0,0)$ is a $\gamma_0$-complete ultrafilter on $\gamma_0$.
\item $\mathcal{D}(n,\alpha)$ is a $\gamma_n$-complete ultrafilter on $\gamma_n$ and if $\alpha\neq\beta$, then $\mathcal{D}(n,\alpha)\neq\mathcal{D}(n,\beta)$.
\end{itemize}
Define $\mathbb{P}_{\mathcal{D}}$ is the forcing notion as follows:
\begin{itemize}
\item Forcing conditions are pairs $(s,F)$, such that $s$ is a finite strictly increasing string of ordinals, $|s|>0$ and $\forall i<|s|(s(i)<\gamma_i)$, and $F$ is a function with the same domain of $\mathcal{D}$ and for each $x\in\dom(F)$, $F(x)\in\mathcal{D}(x)$. Such a function $F$ is called \emph{$\mathcal{D}$-choice function}.
\item $(s,F)$ and $(s',F')$ are conditions, $(s',F')$ is stronger than $(s,F)$ iff
\begin{itemize}
\item[(a)] $|s|\leq|s'|$ and $s'{\uh}|s|=s$.
\item[(b)] For $x\in\dom(\mathcal{D})$, $F'(x)\subset F(x)$.
\item[(c)] $\forall i(|s|\leq i<|s'|\rightarrow s'(i)\in F(i,s'(i-1)))$.
\end{itemize}
\end{itemize}
\end{definition}

\begin{remark}
Such $(\langle\gamma_n:n<\omega\rangle,\mathcal{D})$ can be obtained by iterated forcing from $\omega$ many measurable cardinals.
\end{remark}

Every condition in $\mathbb{P}_{\mathcal{D}}$ corresponds to a subtree of $U=\{s\colon |s|>0\ \&\ \forall i<|s|\ s(i)<\gamma_i\}$.
For $(s,F)\in\mathbb{P}_{\mathcal{D}}$, $|s|=n$ and $\tilde{s}=\langle s(0),...,s(n-1)\rangle$, define $T(s,F)$ as follows:
\[T(s,F)=\{t\in U\colon (t=s{\uh}|t|)\ \vee\ (t{\uh}n=s\ \&\ \forall i(n\leq i<|t|\rightarrow t(i)\in F(i,t(i-1))))\}.\]

Similar to the case for $\mathbb{P}_{\vec{\mu}}$,
$\{T(s,F)\colon(s,F)\in\mathbb{P}_{\mathcal{D}}\}$ is isomorphic to the separative quotient of $\mathbb{P}_{\mathcal{D}}$, we use $\mathbb{P}_{\mathcal{D}}^{'}$ to denote this forcing notion.

\begin{lemma}
Suppose $G$ is a $\mathbb{P}_{\mathcal{D}}$-generic filter over $V$, define $g=\bigcup\{s\colon\exists F\ (s,F)\in G\}$. Then $g\colon\omega\rightarrow\gamma$ and $\forall i<\omega\ (g(i)<\gamma_i)$. We call $g$, the generic sequence corresponding to $G$. Also, we have
\[G=\{(s,F)\in\mathbb{P}_{\mathcal{D}}\colon g{\uh}|s|=s\ \&\ \forall i(|s|\leq i\rightarrow g(i)\in F(i,g(i-1)))\}.\]
So $V[G]=V[g]$. $g$ is a branch of $T(s,F)$ for $(s,F)\in G$.

If $x\colon\omega\rightarrow\gamma$ such that $\forall i\ (x(i)<\gamma_i)$ and $x\in V$, then there is $m<\omega$, if $n>m$, $x(n)<g(n)<\gamma_n$. In particular, there is $m<\omega$, if $n>m$, $\gamma_{n-1}<g(n)<\gamma_n$.
\end{lemma}

\begin{proof}
Just to prove the second statement. Given $x$. Define
\begin{equation*}
\begin{split}
E_x=\{(s,F)\in\mathbb{P}_{\mathcal{D}}\colon\forall t\in U(t{\uh}|s|=s\ \&\ \forall i\geq |s|(t(i)\in F&(i,t(i-1)))\rightarrow \\ \forall& i\geq|s|(t(i)>x(i)))\}.
\end{split}
\end{equation*}
Then $E_x$ is dense in $\mathbb{P}_{\mathcal{D}}$. So $G\cap E_x\neq\emptyset$. So $m<\omega$, if $n>m$, $x(n)<g(n)<\gamma_n$.
\end{proof}

\begin{notation}
For $p\in\mathbb{P}_{\mathcal{D}}^{''}$, $x\in p$, $h(x)$ is the height of $x$ in the tree $p$, so $h(x)\in\omega$. $h(x)=0$ iff $x<\gamma_0$, for $i>0$, $h(x)=i$ iff $\gamma_{i-1}<x<\gamma_i$. $\stem(p)$, $p^{(x)}$, $p(n)$ etc. Notations are similar to corresponding parts for $\mathbb{P}_{\mathcal{D}}$, see Notation \ref{notationfortree}.
\end{notation}

Given $(k,\alpha)\in\dom(\mathcal{D})$, let us define a sequence $\langle E_n^{(k,\alpha)}\colon n<\omega\rangle$ as follows:
\begin{itemize}
\item $E_n^{(k,\alpha)}$ is an ultrafilter on $\gamma_k\times...\times\gamma_{k+n-1}$, $E_1^{(k,\alpha)}=\mathcal{D}(k,\alpha)$;
\item $E_{n+1}^{(k,\alpha)}=E_n^{(k,\alpha)}*\langle\mathcal{D}(k+n-1,\beta)\colon\beta<\gamma_{k+n-2}\rangle$.
\end{itemize}

If $p\in\mathbb{P}_{\mathcal{D}}^{'}$, then $E_n^p$ is just $E_n^{(k,\alpha)}$ where $k=\stem(p)$ and $\{t\colon\stem(p)^{\frown}t\in p\}\in\mathcal{D}(k,\alpha)$. $E_n^p$ is just related with $\stem(p)$. If we have fixed $\alpha$ or $p$, then $E_n$ has no confusion. Moreover, we have:

\begin{proposition}
$p\in\mathbb{P}_{\mathcal{D}}^{'}$. Then
\begin{itemize}
\item[(1)] For $n>0$, $p(n)\in E_n^{p}$.
\item[(2)] $q$ is a subtree of $p$ and $\stem(q)=\stem(p)$, then $q\in\mathbb{P}_{\mathcal{D}}^{'}$ iff for $n>0$, $q(n)\in E_n^p$.
\item[(3)] If $A\subset p(n)$, define $p^A=\bigcup_{x\in A}p^{(x)}$. Then $A\in E_n$ iff $p^{A}\in\mathbb{P}_{\mathcal{D}}^{'}$.
\end{itemize}
\end{proposition}

\begin{corollary}
Given $(k,\alpha)\in\dom(\mathcal{D})$, for a function $f$ with domain $\gamma_k\times...\times\gamma_{k+n-1}$, $f$ is $E_n^{(k,\alpha)}$-equivalent to a function of the following types:
\begin{itemize}
\item[(0)] A constant function.
\item[(1)] A function $g$ defined from a function with domain $\gamma_k$, i.e. there is $h$ with domain $\gamma_k$, $g(x_0,...,x_{i-1})=h(x_0)$ for all $x_0,...,x_{i-1}$.
\item[...]
\item[(i)] A function $g$ defined from a function with domain $\gamma_k\times...\times\gamma_{i+n-1}$, i.e. there is $h$ with domain $\gamma_k\times...\times\gamma_{i+n-1}$, $g(x_0,...,x_{i-1})=h(x_0,...,x_{k-1})$ for all $x_0,...,x_{i-1}$.
\item[...]
\item[(n)] An one-to-one function with domain $\gamma_k\times...\times\gamma_{k+n-1}$.
\end{itemize}
\end{corollary}

\begin{proof}
This follows from Corollary \ref{hanshufenlei2}.
\end{proof}

The following basic theorems for Prikry-type forcing has similar proofs as $\mathbb{P}_{\vec{\mu}}$, so we omit the proof here.
\begin{theorem}[Prikry property]
Suppose $\varphi$ is a sentence in the language of the forcing $\mathbb{P}_{\mathcal{D}}$, given $s$. Then there is $F$, a $\mathcal{D}$-choice function, $(s,F)$ decides $\varphi$.
\end{theorem}

\begin{theorem}[Geometric condition]\label{diaggeometric}
$M$ is a transitive model of $\ZFC$, in $M$, $(\langle\gamma_n:n<\omega\rangle,\mathcal{D})$ satisfies the hypothesis above.
$S\colon\omega\rightarrow\gamma$ such that $\forall n<\omega\ S(n)<\gamma_n$. Then $S$ is a $(\mathbb{P}_{\mathcal{D}})^{M}$-generic sequence over $M$ iff for each $\mathcal{D}$-choice function $F$ in $M$, there is $0<n<\omega$, such that $\forall i(n\leq i<\omega\rightarrow S(i)\in F(i,S(i-1)))$.
\end{theorem}

Given $(m,\theta)\in\dom(\mathcal{D})$, define \[\mathbb{P}_{\mathcal{D}}^{(m,\theta)}=\mathbb{P}_{\mathcal{D}_1}\mbox{ where }\mathcal{D}_1=\mathcal{D}{\uh}(\{(m,\theta)\}\cup\{(n,\alpha)\colon n>m\ \&\ \alpha<\gamma_{n-1}\}).\]
Then $\mathbb{P}_{\mathcal{D}}^{(m,\theta)}$ has the property of $\mathbb{P}_{\mathcal{D}}$ which is proved for $(\langle\gamma_n:n<\omega\rangle,\mathcal{D})$. Then following fact is a direct corollary of geometric condition.

\begin{proposition}
\begin{itemize}
\item[(1)] If $g$ is a $\mathbb{P}_{\mathcal{D}}$-generic sequence, then $\langle g(i): i\geq m\rangle$ is a $\mathbb{P}_{\mathcal{D}}^{(m,\theta)}$-generic sequence.
\item[(2)] If $h$ is a $\mathbb{P}_{\mathcal{D}}^{(m,\theta)}$-generic sequence and $|s|=m$ and $\forall i<m\ (s(i)<\gamma_i)$, then $s^{\frown}h$ is a $\mathbb{P}_{\mathcal{D}}$-generic sequence.
\end{itemize}
Thus a $\mathbb{P}_{\mathcal{D}}$-generic extension is a $\mathbb{P}_{\mathcal{D}}^{(m,\theta)}$-generic extension, the converse is also true.
\end{proposition}

Using this result, let us prove $\mathbb{P}_{\mathcal{D}}$ does not add new bounded subset of $\gamma$. Note that for the corresponding result for $\mathbb{P}_{\vec{\mu}}$, geometric condition is not used. An advantage of $\mathbb{P}_{\mathcal{D}}^{(m,\theta)}$ is that for less than $\gamma_{m+1}$ many condition of $\mathbb{P}_{\mathcal{D}}^{(m,\theta)}$ with the same stem, there is a condition stronger than them.

\begin{lemma}
\begin{itemize}
\item[(1)] $\mathbb{P}_{\mathcal{D}}$ does not add new bounded subset of $\gamma$.
\item[(2)] $\mathbb{P}_{\mathcal{D}}$ has $\gamma^{+}$-c.c.
\item[(3)] $\mathbb{P}_{\mathcal{D}}$ preserves all cofinalities and cardinals.
\end{itemize}
\end{lemma}

\begin{proof}
(1) Suppose $G$ is a $\mathbb{P}_{\mathcal{D}}$-generic filter over $V$. In $V[G]$, $A$ is a bounded subset of $\gamma$. $A\subset\gamma_n$. Let $\theta<\gamma_{n-1}$. Then $V[G]$ is also a generic extension for $\mathbb{P}_{\mathcal{D}}^{(n,\theta)}$. $\dot{A}$ is a $\mathbb{P}_{\mathcal{D}}^{(n,\theta)}$-name for $A$ such that $\Vdash_{\mathbf{R}^{(n,\theta)}}\dot{A}\subset\gamma_n$. $\mathcal{D'}=\mathcal{D}{\uh}(\{(n,\theta)\}\cup\{(i,\alpha)\colon i>n\ \&\ \alpha<\gamma_{i-1}\})$.
Given $s$, for $\alpha<\gamma_n$, let $F_{\alpha}$ be a $\mathcal{D'}$-choice function such that $(s,F_{\alpha})$ decides ``$\alpha\in\dot{A}$". Define $H$ as follows: $H(x)=\bigcap_{\alpha<\gamma_n}F_{\alpha}(x)$. By the proposition above, $H$ is a $\mathcal{D'}$-choice function. Then for all $\alpha<\gamma_n$, $(s,H)$ decides ``$\alpha\in\dot{A}$". So there is such a $(s,H)\in G$. Thus $A=\{\alpha<\gamma_n\colon(s,H)\Vdash\alpha\in\dot{A}\}\in V$.

(2) Because $|\gamma^{<\omega}|=\gamma$, and $(s,F_1)$ and $(s,F_2)$ are compatible.

(3) This follows from (1) and (2).
\end{proof}

\subsection{$\mathbb{P}_{\mathcal{D}}$ yields minimal extensions}

Given a $\mathcal{D}$-choice function $K$ satisfies:
\begin{itemize}
\item $K(0,0)=\gamma_0$.
\item For $n>0$, $\langle K(n,\alpha)\colon\alpha<\gamma_{n-1}\rangle$ is a family of disjoint subsets of $\{x\colon\gamma_{n-1}<x<\gamma_n\}$.
\end{itemize}
The existence of $K$ depends on normality of all ultrafilters.

$A=\{(s,F)\in\mathbb{P}_{\mathcal{D}}\colon\forall(n,\alpha)\ F(n,\alpha)\subset K(n,\alpha)\}$ is a dense open subset of $\mathbb{P}_{\mathcal{D}}$. So $B=\{T(s,F)\colon (s,F)\in A\}$ is a dense subset of $\mathbb{P}_{\mathcal{D}}^{'}$. So $B$ with ordering ``subtree", as a forcing notion, is equivalent to $\mathbb{P}_{\mathcal{D}}$ and $\mathbb{P}_{\mathcal{D}}^{'}$. For $p\in B$, by the property of $K$, if $x,y\in p$ and $x\neq y$, then $\max(x)\neq\max(y)$. So we can identify $p$ with a subset of $\gamma$, i.e. using $\max(x)$ to replace $x$ in $p$. We use $\mathbb{P}_{\mathcal{D}}^{''}$ to denote this forcing notion, it is equivalent to $\mathbb{P}_{\mathcal{D}}$ and $\mathbb{P}_{\mathcal{D}}^{'}$.

\begin{lemma}\label{gammaziji}
Suppose $G$ is a $\mathbb{P}_{\mathcal{D}}$-generic filter over $V$. In $V[G]$, $A\subset\gamma$ and $A\notin V$. Then $V[A]=V[G]$.
\end{lemma}

\begin{proof}
Give $G$ and $A$.
Fix $\dot{a}$, a $\mathbb{P}_{\mathcal{D}}$-name, such that $\Vdash``\dot{a}\colon\gamma\rightarrow 2\mbox{ and }\dot{a}\notin V."$ and $A=\{\alpha<\gamma\colon\dot{a}/G(\alpha)=1\}$. Let $c\colon\gamma\rightarrow 2$ be $\dot{a}/G$. So $V[A]=V[c]$. Given $p$ and $\mathcal{A}$, $\Phi(p,\mathcal{A})$ is the conjunction of the following sentences:
\begin{itemize}
\item[(1)] $p\in\mathbb{P}_{\mathcal{D}}^{''}$.
\item[(2)] For all $x\in p$ with $x\geq_p\stem(p)$, $p^{(x)}$ decides the value of $\dot{a}{\uh}\gamma_{h(x)}$, i.e. for $\alpha<\gamma_{h(x)}$, $p^{(x)}\Vdash\dot{a}(\alpha)=0$ or $p^{(x)}\Vdash\dot{a}(\alpha)=1$.
\item[(3)] $\mathcal{A}$ is a function with domain $\subset p$ and $\stem(p)\in\dom(\mathcal{A})$.
\item[\ ]   For $x\in\dom(\mathcal{A})$,
\item[(4)] $\mathcal{A}(x)=(\tilde{k}(x),k(x))$, $1\leq\tilde{k}(x)\leq k(x)<\omega$.
\item[(5)] $p^{(x)}(\tilde{k}(x))\subset\dom(\mathcal{A})$.
\item[(6)] There is $y\leq_p x$, $y\in\dom(\mathcal{A})$ and $x\in p^{(y)}(\tilde{k}(y))$.
\item[(7)] If $y\in p(\tilde{k}(x))$, then $p^{(y)}$ decides the value of $\dot{a}{\uh}(\gamma_{h(x)+k(x)})$.
\item[(8)] If $y,z\in p(\tilde{k}(x))$ and $y\neq z$. $\sigma_1$ and $\sigma_2$ are $0$-$1$ sequences of length $\gamma_{h(x)+k(x)}$ and $p^{(y)}\Vdash\dot{a}{\uh}\gamma_{h(x)+k(x)}=\sigma_1$ and $p^{(z)}\Vdash\dot{a}{\uh}\gamma_{h(x)+k(x)}=\sigma_2$, then $\sigma_1\neq\sigma_2$.
\end{itemize}

It is clear that if $\Phi(p,\mathcal{A})$ holds, $x=\stem(p)$, $\mathcal{A}(x)=(\tilde{k}(x),k(x))$, $y\in p(h(x)+\tilde{k}(x))$, then $\Phi(p^{(y)},\mathcal{A}{\uh}p^{(y)})$.

\begin{claim}
$W=\{q\in\mathbb{P}_{\mathcal{D}}^{''}\colon q\Vdash\dot{a}\in V\mbox{ or }\exists\mathcal{A}\ \Phi(q,\mathcal{A})\}$ is a dense subset of $\mathbb{P}_{\mathcal{D}}^{''}$.
\end{claim}
Given $p\in\mathbb{P}_{\mathcal{D}}^{''}$, let us define a subtree $r$ of $p$, $r\in W$.

\noindent \underline{\emph{Step 1.}} There is a pure extension $p'$, $p'$ satisfies for all $x\in p'$ with $x\geq_{p'}\stem(p')$, ${p'}^{(x)}$ decides the value of $\dot{a}{\uh}\gamma_{h(x)}$. Construction as follows:

$z=\stem(p)$. Then $p(1)\in\mathcal{D}(h(z)+1,z)$. $\mathcal{D}(h(z)+1,z)$ is a $\gamma_{h(z)+1}$-complete ultrafilter on $\gamma_{h(z)+1}$. $\gamma_{h(z)}<\gamma_{h(z)+1}$, so $p$ can be extended purely to a condition $\tilde{p}$ decides the value of $\dot{a}{\uh}\gamma_{h(z)}$. For element of $\tilde{p}(1)$ (replace the position of $z$), repeat the process above, inductively, the pure extension $p'$ is defined.

\noindent \underline{\emph{Step 2.}} In step 1, essentially, we have define functions $f_n$ for $0<n<\omega$: $\dom(f_n)=p'(n)$, for $x\in\dom(f_n)$, $f_n(x)$ is the value of $\dot{a}{\uh}\gamma_{h(x)}$ decided by ${p'}^{(x)}$. We have two cases now:

\noindent \emph{Case 1.} For each $0<n<\omega$, $f_n$ is $E_n$-equivalent to a constant function. In this case, there is $p''$, a pure extension of $p'$, such that for any $n$, and $x,y\in p''(n)$, the value of $\dot{a}{\uh}(|\stem(p'')|+n)$ decided by ${p''}^{(x)}$ and ${p''}^{(y)}$ are the same. So from the tree $p''$, $\dot{a}/G$ is defined: since for any branch $g$ of $p''$, if $g$ is the generic sequence corresponding to $G$, $\dot{a}/G$ are all the union of values of $f_n$'s.

\noindent \emph{Case 2.} There exists $0<n<\omega$, $\rank_n(f_n)=m>0$. In this case, we obtain two numbers: $n$ and $m$ such that $0<m\leq n$. For $x\in p'(m)$, if $y,z\in p'(n)$ and $y,z\geq_{p'}x$, then the value of $\dot{a}{\uh}(|\stem(p')|+n)$ decided by ${p'}^{(y)}$ and ${p'}^{(z)}$ are the same. So in fact, ${p'}^{(x)}$ decides the value of $\dot{a}{\uh}(|\stem(p')|+n)$.

\noindent \underline{\emph{Step 3.}} Repeat the process of step 2 from the stem of $p'$ by induction on $n$. If in the whole process, case 1 of step 2 does not occur, then we obtain a subtree $q$ of $p'$ and a function $\mathcal{A}$ such that $\Phi(q,\mathcal{A})$: in each position $(m,n)$ is the value of $\mathcal{A}$. Otherwise, at some $x\in p'$, case 1 of step 2 occurs, then ${p'}^{(x)}$ can be purely extended to a condition who can compute the value of the whole $\dot{a}$, so $p'$ can be extended (not need to be purely) to a condition who forces $\dot{a}\in V$.

We proved the claim.

Since $G$ is a $\mathbb{P}_{\mathcal{D}}^{''}$-generic filter over $V$, $G\cap W\neq\emptyset$.
If $q\in G\cap W$ and $q\Vdash\dot{a}\in V$, then $c\in V$. So $A\in V$, a contradiction.

If $q\in G\cap W$ such that $\exists\mathcal{A}\ \Phi(q,\mathcal{A})$. Let us prove $G\in V[c]$. $g$ is the generic sequence corresponding to $G$. Just need to prove $g\in V[c]$. Work in $V[c]$. From $q$, $\mathcal{A}$ and $c$, we can compute $g$ as follows:

Since $q\in G$, so $g$ must be a branch of $q$. Define a branch $e\colon\omega\rightarrow q$ as follows: $e$ extends the stem of $q$. $x=\stem(q)$, $\mathcal{A}(x)=(\tilde{k}(x),k(x))$. By (7) above, if $y\in q(\tilde{k}(x))$, then $q^{(y)}$ decides the value of $\dot{a}{\uh}(\gamma_{h(x)+k(x)})$. By (8), there is exactly one $y$ such that the value of $\dot{a}{\uh}(\gamma_{h(x)+k(x)})$ decided by $q^{(y)}$ is equal to $c{\uh}(\gamma_{h(x)+k(x)})$. Then $e(h(y))=y$. So we have define $e$ to the $h(y)$-th position. For $q^{(y)}$ and $\mathcal{A}$, since $\Phi(q^{(y)},\mathcal{A}{\uh}q^{(y)})$, so we can repeat this process, and then extend $e$ longer and longer. Obviously, $e=g$. So $g\in V[c]$. So $V[c]=V[G]$. So $V[A]=V[G]$.
\end{proof}

\begin{theorem}\label{Rminimal}
Forcing notion $\mathbb{P}_{\mathcal{D}}$ yields minimal extensions.
\end{theorem}

\begin{proof}
Using lemma \ref{gammaziji}, just replace all $\kappa$ in the proof of theorem \ref{Qminimal} by $\gamma$.
\end{proof}

\section{Iterated ultrapowers}

\subsection{Intermediate submodels}

Let us recall some concepts. An \emph{elementary chain} of length $\tau$, where $\tau$ is an ordinal or $\tau=\Ord$, is a system $\langle M_{\alpha},j_{\alpha,\beta}\colon\alpha\leq\beta<\tau\rangle$ such that
\begin{itemize}
\item $M_{\alpha}$'s are transitive models of $\ZFC$.
\item $j_{\alpha,\beta}$'s are elementary embeddings and are commutative.
\item If $\alpha$ is a limit ordinal less than $\tau$, then $M_{\alpha}$ is the direct limit of $\langle M_{\beta}\colon \beta<\alpha\rangle$.
\end{itemize}
An \emph{iterated ultrapower} is a pair $(\langle M_{\alpha},j_{\alpha,\beta}\colon\alpha\leq\beta<\tau\rangle,\langle(\kappa_{\alpha},U_{\alpha})\colon\alpha<\tau\rangle)$ such that
\begin{itemize}
\item $\langle M_{\alpha},j_{\alpha,\beta}\colon\alpha\leq\beta<\tau\rangle$ is an elementary chain.
\item For each $\alpha<\tau$, $\kappa_{\alpha}$ is an infinite cardinal in $M_{\alpha}$ and $U_{\alpha}$ is an $M_{\alpha}$-ultrafilter over $\kappa_{\alpha}$.
\item $M_{\alpha+1}=\Ult(M_{\alpha},U_{\alpha})$.
\end{itemize}

\begin{definition}
$N$ and $M$ are transitive models, $j\colon N\rightarrow{M}$ is an elementary embedding,
\begin{itemize}
\item[(1)] Given a class $B\subset{M}$, define
\[\mathcal{H}^{j}(B)=\{j(f)(s)\colon f\in{N},(f\mbox{ is a function})^{N},s\in[B]^{<\omega},s\in\dom(j(f))\}.\]
\item[(2)] A class $C\subset M$ is called an \emph{intermediate submodel} if $\forall x\in N(j(x)\in C)$ and $C\prec M$.
\end{itemize}
\end{definition}

Clearly, $j''N$ is an intermediate submodel, all intermediate submodels are well-founded. The following facts are essentially appears in \cite{3}.

\begin{proposition}
\begin{itemize}
\item[(1)] $B\subset\mathcal{H}^{j}(B)$ and $j''N\subset\mathcal{H}^{j}(B)$.
\item[(2)] $\mathcal{H}^{j}(B)$ is the least intermediate submodel which contains $B$.
\item[(3)] $\mathcal{H}^{j}(B)=\mathcal{H}^{j}(B{\setminus}j''N)$.
\end{itemize}
\end{proposition}

\begin{proposition}
Suppose a class $C\subset M$, then the followings are equivalent:
\begin{itemize}
\item[(a)] $C$ is an intermediate submodel.
\item[(b)] $C=\mathcal{H}^{j}(B)$ for some class $B\subset M$.
\item[(c)] $C=\mathcal{H}^{j}(C)$.
\end{itemize}
\end{proposition}

\subsection{$(*)$-iterated ultrapower and diagonal $(*)$-iterated ultrapower}

Let us define two classes of iterated ultrapowers:
\begin{definition}
Given an iterated ultrapower $(\langle M_n,j_{n,m}\colon n\leq m\leq\omega\rangle, \langle\kappa_n,U_n\colon n<\omega\rangle)$,
\begin{itemize}
\item[(1)] We call it a \emph{$(*)$-iterated ultrapower}, if for each $n<\omega$, in $M_n$, $U_n$ is a normal measure on $\kappa_n$, $\kappa_{n+1}=j_{n,n+1}(\kappa_n)$, $U_{n+1}\notin j_{n,n+1}''M_n$.
\item[(2)] We call it a \emph{diagonal $(*)$-iterated ultrapower}, if in $M_0$, $\langle\kappa_n\colon n<\omega\rangle$ is a strictly increasing sequence of inaccessible cardinals, for each $n<\omega$, in $M_n$, $U_n$ is a normal measure on $\kappa_n$ and $U_{n+1}\notin j_{n,n+1}''M_n$.
\end{itemize}
\end{definition}

\begin{remark}
In diagonal $(*)$-iteration case, in fact, $\langle\kappa_n\colon n<\omega\rangle$ is a strictly increasing sequence of measurable cardinals.
\end{remark}

\begin{lemma}
$(\langle M_n,j_{n,m}\colon n\leq m\leq\omega\rangle, \langle\kappa_n,U_n\colon n<\omega\rangle)$ is a diagonal $(*)$-iterated ultrapower, then
\begin{itemize}
\item[(1)] If $n\leq m<\omega$, $j_{0,n}(\kappa_m)=\kappa_m$.
\item[(2)] If $\kappa=\sup\{\kappa_n\colon n<\omega\}$, then $j_{0,\omega}(\kappa)=\kappa$.
\end{itemize}
\end{lemma}

\begin{proof}
(1) Note that the following fact(\cite{2}): $\mu$ is a normal measure on $\lambda$, $j\colon V\rightarrow\Ult(V,\mu)$ is the induced elementary embedding, then
\begin{itemize}
\item For ordinal $\alpha$, if $\cf(\alpha)>\lambda$, then $j$ is continuous at $\alpha$, i.e. $\sup(j''\alpha)=j(\alpha)$.
\item For ordinal $\alpha$, $|j(\alpha)|<(|\alpha|^{\kappa})^{+}$.
\end{itemize}
For $i<m$, $\cf(\kappa_m)=\kappa_m>\kappa_i$, so $j_{i,i+1}$ is continuous at $\kappa_m$. In $M_i$, for $\eta<\kappa_m$, \[|j_{i,i+1}(\eta)|<|\eta^{\kappa_i}|^{+}<\kappa_m.\] So $j_{i,i+1}(\kappa_m)=\kappa_m$. So for $n\leq m$, $j_{0,n}(\kappa_m)=\kappa_m$.

(2) Because $j_{0,\omega}(\langle\kappa_n\colon n<\omega\rangle)=\langle j_{0,\omega}(\kappa_n)\colon n<\omega\rangle$, and $\forall n<\omega\ \kappa_n\leq j_{0,\omega}(\kappa_n)<\kappa_{n+1}$. Thus we have $j_{0,\omega}(\kappa)=\kappa$.
\end{proof}

\begin{theorem}
Suppose $(\langle M_n,j_{n,m}\colon n\leq m\leq\omega\rangle, \langle\kappa_n,U_n\colon n<\omega\rangle)$ is a $(*)$-iterated ultrapower or diagonal $(*)$-iterated ultrapower. If $X$ is an intermediate submodel of $j_{0,\omega}\colon M_0\rightarrow M_{\omega}$, then $X=M_{\omega}$ or there is $i<\omega$, $X=j_{i,\omega}''M_i$.
\end{theorem}

\begin{proof}
Given $(\langle M_n,j_{n,m}\colon n\leq m\leq\omega\rangle, \langle\kappa_n,U_n\colon n<\omega\rangle)$, the situation described in corollary \ref{hanshufenlei2} occurs: if it is a $(*)$-iterated ultrapower, for each $n<\omega$, $\lambda_n=\kappa$, $\pi_0$ is a normal measure on $\kappa$ and for $i>0$, $\pi_i$ is a represent function for $U_i$ in the ultrapower $\Ult(M_0,D_i)$; if it is a diagonal $(*)$-iterated ultrapower, $\lambda_n=\kappa_n$, $\pi_0$ is a normal measure on $\kappa_0$ and for $i>0$, $\pi_i$ is a represent function for $U_i$ in the ultrapower $\Ult(M_0,D_i)$.
By corollary \ref{hanshufenlei2}, for a function $f$ with domain $\lambda_0\times...\times\lambda_{i-1}$, $f$ is $D_n$-equivalent to a function of the following types:
\begin{itemize}
\item[(0)] A constant function.
\item[(1)] A function $g$ defined from a function with domain $\lambda_0$, i.e. there is $h$ with domain $\lambda_0$, $g(x_0,...,x_{i-1})=h(x_0)$ for all $x_0,...,x_{i-1}$.
\item[...]
\item[(k)] A function $g$ defined from a function with domain $\lambda_0\times...\times\lambda_{k-1}$, i.e. there is $h$ with domain $\lambda_0\times...\times\lambda_{k-1}$, $g(x_0,...,x_{i-1})=h(x_0,...,x_{k-1})$ for all $x_0,...,x_{i-1}$.
\item[...]
\item[(i)] An one-to-one function with domain $\lambda_0\times...\times\lambda_{i-1}$.
\end{itemize}

For $a\in M_{\omega}$, because $M_{\omega}=\bigcup_{i<\omega}j_{i,\omega}''M_i$, let $n$ be the least natural number such that $a\in j_{n,\omega}''M_n$. $M_n=\Ult(M_0,D_n)$, so there is $f$ in $M_0$ with domain $\lambda_0^{n}$, $[f]_{D_n}=j_{n,\omega}^{-1}(a)$. From the argument above and minimality of $n$, there is an one-to-one function $g$ which is $D_n$-equivalent to $f$. So $[g]_{D_n}=j_{n,\omega}^{-1}(a)$.
By normality of every value of $\pi_n$, $[g]_{D_n}=j_{0,n}(g)(\lambda_0,...,\lambda_{n_1})$.
So $(\lambda_0,...,\lambda_{n-1})=j_{0,n}(g^{-1})(j_{n,\omega}^{-1}(a))$.
So
\begin{equation*}
\begin{split}
j_{0,\omega}(g^{-1})(a)=j_{n,\omega}(j_{0,n}(g^{-1}))(a)=j_{n,\omega}(j_{0,n}(g^{-1})(j_{n,\omega}^{-1}(a)))=
j_{n,\omega}&(\lambda_0,...,\lambda_{n-1})\\
&=(\lambda_0,...,\lambda_{n-1}).
\end{split}
\end{equation*}

Given $X$, an intermediate submodel of $j_{0,\omega}$.

\noindent\emph{Case 1.} There is $n<\omega$, $X\subset j_{n,\omega}''M_n$.

There is $a\in X$, the least $k$ such that $a\in j_{k,\omega}''M_k$ is $n$. From the argument above, there is $g$, such that $j_{0,\omega}(g^{-1})(a)=(\lambda_0,...,\lambda_{n-1})$. So $(\lambda_0,...,\lambda_{n-1})\in X$. But since $M_n=\{j_{0,n}(f)(\lambda_0,...,\lambda_{n-1})\colon f\in M_0\}$, so $j_{n,\omega}''M_n=\{j_{0,\omega}(f)(\lambda_0,...,\lambda_{n-1})\colon f\in M_0\}$. So $j_{n,\omega}''M_n\subset X$. Thus $j_{n,\omega}''M_n=X$.

\noindent\emph{Case 2.} There is no $n<\omega$, $X\subset j_{n,\omega}''M_n$.

From the argument above, in this case, for each $i<\omega$, $\kappa_i\in X$. So $X=M_{\omega}$.
\end{proof}

\subsection{$\mathbb{P}_{\vec{\mu}}$ and $(*)$-iterated ultrapower}

Let us prove $\mathbb{P}_{\vec{\mu}}$ yields minimal extensions using the analysis of intermediated submodels of $(*)$-iterated ultrapower.

$\kappa$ and $\vec{\mu}$ are as above, $\alpha<\kappa$. Consider iterated ultrapower \[\mathcal{A}(\alpha)=(\langle M_n,j_{n,m}\colon n\leq m\leq\omega\rangle, \langle\kappa_n,U_n\colon n<\omega\rangle)\] satisfies:
\begin{itemize}
\item $M_0=V$, $\kappa_0=\kappa$, $U_0=\mu_{\alpha}$.
\item For each $n<\omega$, $\kappa_{n+1}=j_{n,n+1}(\kappa_n)$ and $U_{n+1}=j_{0,n+1}(\vec{\mu})(\kappa_n)$.
\end{itemize}

\begin{proposition}\label{diedaixingzhi}
\begin{itemize}
\item[(1)] $\mathcal{A}(\alpha)$ exists and is unique.
\item[(2)] For each $n<\omega$, $\kappa_n=\crit(j_{n,n+1})$, $\mathcal{A}(\alpha)$ is a $(*)$-iteration.
\item[(3)] $\langle\kappa_n\colon n<\omega\rangle$ is strictly increasing with supremum $j_{0,\omega}(\kappa)$.
\item[(4)] $\langle\kappa_n\colon n<\omega\rangle$ is a $(\mathbb{P}_{j_{0,\omega}(\vec{\mu})})^{M_{\omega}}$-generic sequence over $M_{\omega}$.
\end{itemize}
\end{proposition}

\begin{proof}
(1)Uniqueness is obvious. Every $M_n$ is just $\Ult(M_0,D_n)$ for some ultrafilter $D_n$ defined by sum of ultrafilters.

(2)Suppose $\mathcal{A}(\alpha)$ is not a $(*)$-iteration. There is $n$, $j_{0,n+1}(\vec{\mu})(\kappa_n)=j_{n,n+1}(x)$, then \[j_{n,n+1}(j_{0,n}(\vec{\mu}))(\kappa_n)=j_{n,n+1}(x).\] So there is $t$ such that $j_{0,n}(\vec{\mu})(t)=x$. Since $j_{0,n}(\vec{\mu})$ is one-to-one, so $j_{n,n+1}(t)=\kappa_n$, this contradicts $\kappa_n=\crit(j_{n,n+1})$.
(3) is clear.

(4)$F\in M_{\omega}$ is a $j_{0,\omega}(\vec{\mu})$-choice function in $M_{\omega}$. Then there is $0<m<\omega$ and $K\in M_m$, a $j_{0,m}(\vec{\mu})$-choice function in $M_m$ such that $j_{m,\omega}(K)=F$. For $i\geq m$, $j_{m,i}(K)$ is a $j_{0,i}(\vec{\mu})$-choice function in $M_i$ and $U_i=j_{0,i}(\vec{\mu})(\kappa_{i-1})$, so $\kappa_i\in j_{i,i+1}(j_{m,i}(K)(\kappa_{i-1}))=j_{m,i+1}(K)(\kappa_{i-1})$. Thus \[\kappa_i=j_{i+1,\omega}(\kappa_i)\in j_{i+1,\omega}(j_{m,i+1}(K)(\kappa_{i-1}))=j_{m,\omega}(K)(\kappa_{i-1})=F(\kappa_{i-1}).\]
From geometric condition, (4) is proved.
\end{proof}

Let us prove $\mathbb{P}_{\vec{\mu}}$ yields minimal extensions using $\mathcal{A}(\alpha)$.
\begin{lemma}
In the model $M_{\omega}[\langle\kappa_n\colon n<\omega\rangle]$, $a\colon\omega\rightarrow j_{0,\omega}(\kappa)$, and $\ran(a)$ is unbounded in $j_{0,\omega}(\kappa)$. Then $M_{\omega}[a]=M_{\omega}[\langle\kappa_n\colon n<\omega\rangle]$.
\end{lemma}

\begin{proof}
Let $X=\mathcal{H}^{j_{0,\omega}}(\ran(a))$. Then $X\prec M_{\omega}$.
\begin{claim}
$X=M_{\omega}$.
\end{claim}
Suppose not. Since the iteration $\langle M_n\colon n<\omega\rangle$ is a $(*)$-iteration, there is $n$, $X=j_{n,\omega}M_n$. So for all $t\in\ran(a)$, there is $y\in M_n$, $j_{n,\omega}(y)=t$. Because $t<\j_{0,\omega}(\kappa)$, so $y<\kappa_n$, so $t=j_{n,\omega}(y)=y$. So $\ran(a)\subset\kappa_n$, this contradicts that $\ran(a)$ is unbounded in $j_{0,\omega}(\kappa)$.

In $V$, there are $\langle f_i\colon i<\omega\rangle$ and $\langle s_i\colon i<\omega\rangle$ such that \[\forall i<\omega(s_i\in[\ran(a)]^{<\omega}\ \&\ \kappa_i=j_{0,\omega}(f_i)(s_i)).\]

Define $h\colon\omega\rightarrow[\omega]^{<\omega}$ by $k\in h(n)$ iff $a(k)\in s_n$. Then $h$ can be coded by a subset of $\omega$, since $P(\omega)^{V}=P(\omega)^{M_{\omega}}$, so $h\in M_{\omega}$. Also,
\[\langle j_{0,\omega}(f_i)\colon i<\omega\rangle=j_{0,\omega}(\langle f_i\colon i<\omega\rangle)\in M_{\omega}.\]
In $M_{\omega}[a]$, from $a$, $\langle f_i\colon i<\omega\rangle$ and $h$, we can define $\langle\kappa_n\colon n<\omega\rangle$. So $\langle\kappa_n\colon n<\omega\rangle\in M_{\omega}[a]$.
\end{proof}

\begin{lemma}
Suppose $G$ is a $\mathbb{P}_{\vec{\mu}}$-generic filter over $V$. In $V[G]$, $a\colon\omega\rightarrow\kappa$ such that $\ran(a)$ is unbounded in $\kappa$. Then $V[a]=V[G]$.
\end{lemma}

\begin{proof}
Define $\Psi$ is the sentence: ``if $a\colon\omega\rightarrow\kappa$ is unbounded in $\kappa$, then $V[a]=V[G]$."
Let us prove that $\Vdash\Psi$. Suppose not. $(s,F)\in\mathbb{P}_{\vec{\mu}}$ such that $(s,F)\Vdash\neg\Psi$, $\alpha=\max(s)$. Consider the iterated ultrapower $\mathcal{A}(\alpha)$. Then \[j_{0,\omega}((s,F))\Vdash_{M_{\omega},j_{0,\omega}(\mathbb{P}_{\vec{\mu}})}\neg\Psi.\] $\langle\kappa_n\colon n<\omega\rangle$ is $j_{0,\omega}(\mathbb{P}_{\vec{\mu}})$-generic sequence over $M_{\omega}$. Define $g=s^{\frown}\langle\kappa_n\colon n<\omega\rangle$ is also $j_{0,\omega}(\mathbb{P}_{\vec{\mu}})$-generic sequence over $M_{\omega}$. $G$ is the corresponding generic filter to $g$, then $(s,F)\in G$. So $M_{\omega}[G]\vDash\neg\Psi$. Since $M_{\omega}[G]=M_{\omega}[g]=M_{\omega}[\langle\kappa_n\colon n<\omega\rangle]$, so $M_{\omega}[\langle\kappa_n\colon n<\omega\rangle]\vDash\neg\Psi$. This contradicts the lemma above.
\end{proof}

\begin{lemma}
Suppose $G$ is a $\mathbb{P}_{\vec{\mu}}$-generic filter over $V$. In $V[G]$, $A\subset\kappa$ and $A\notin V$. $g$ is the corresponding generic sequence of $G$. Then there is a sequence $\langle\alpha_n\colon n<\omega\rangle$ such that
\begin{itemize}
\item[(1)] For all $n$, $\alpha_n\geq g(n)$.
\item[(2)] $V[A]=V[\langle\alpha_n\colon n<\omega\rangle]$.
\end{itemize}
\end{lemma}

\begin{proof}
Suppose $\dot{A}$ is a $\mathbb{P}_{\vec{\mu}}$-name for $A$ such that $\Vdash\dot{A}\subset\kappa\ \&\ \dot{A}\notin V$.
Let us construct $\langle\alpha_n\colon n<\omega\rangle$ in $V[A]$.
\begin{claim}
Suppose $p\in\mathbb{P}_{\vec{\mu}}^{'}$. Then there is a pure extension $q$ and $\langle a_x\colon x\in q\rangle$ such that
\begin{itemize}
\item[(1)] For $x\in q$ and $x\geq_q\stem(q)$, $q^{(x)}\Vdash\dot{A}\cap\max(x)=a_x$.
\item[(2)] For $t\in q$ and $t\geq_q\stem(q)$, $\alpha<\beta$ such that $t_1=t^{\frown}\alpha$ and $t_2=t^{\frown}\beta$ are in $q$, then $a_{t_2}\cap\alpha=a_{t_1}$.
\end{itemize}
\end{claim}
To satisfies (1), we use $\kappa$-completeness of each ultrafilter. To satisfies (2), the following fact on normal measure is used: $U$ is a normal measure on an infinite cardinal $\kappa$, $\langle b_{\alpha}\colon \alpha<\kappa\rangle$ such that for $\alpha<\kappa$, $b_{\alpha}\subset\alpha$, then there is $K\in U$ such that if $\beta,\gamma\in K$ and $\beta<\gamma$, then $b_{\gamma}\cap\beta=b_{\beta}$ (\cite{22}).

Pick $q\in G$ satisfies the conditions in claim. $n=|\stem(q)|$. For $x\in q$ with $|x|>n$, $A(x)=\bigcup\{a_{x^{\frown}\alpha}\colon x^{\frown}\alpha\in q\}$. Then $A(x)\subset\kappa$ and $A(x)\in V$. Since $A\notin V$, so $A\neq A(x)$ for all $x$. Define $\theta_x$ be the least $\beta$ such that $A\cap\beta\neq A(x)\cap\beta$. We define $\langle\alpha_n\colon n<\omega\rangle$ as follows:
\begin{itemize}
\item $\langle\alpha_n\colon n<\omega\rangle{\uh}n=\stem(q)$.
\item $\alpha_{n}=\theta_{\stem(q)}$.
\item If $\alpha_i$ has defined, then $\alpha_{i+1}=\sup\{\theta_x\colon x\in q(i-n+1)\ \&\ \max(x)<\alpha_i\}$.
\end{itemize}

Then For all $n$, $\alpha_n\geq g(n)$. The reason is: suppose for some $n$, $g(n)\geq\alpha_n$. Since $q\in G$, $\dot{A}/G\neq A$, this contradicts $\dot{A}$ is a name for $A$. Because our construction is in $V[A]$, so $\langle\alpha_n\colon n<\omega\rangle\in V[A]$.

Because $g$ is cofinal in $\kappa$, so $\langle\alpha_n\colon n<\omega\rangle$ is also cofinal in $\kappa$. Note that the sequence $\langle\alpha_n\colon n<\omega\rangle$ need not be increasing.

Let us prove $A\in V[\langle\alpha_n\colon n<\omega\rangle]$. Just need to in $V[\langle\alpha_n\colon n<\omega\rangle]$ to define the sequence $\langle A\cap\alpha_n\colon n<\omega\rangle$. For any $m$, there is $i$, such that $\alpha_m<\alpha_i$. Then in $q(i-n+1)$, there is $x$, $\max(x)>\alpha_m$. So $A\cap\alpha_m=a_x\cap\alpha_m$. So we define $A\cap\alpha_m$. So we have defined $\langle A\cap\alpha_n\colon n<\omega\rangle$ in $V[\langle\alpha_n\colon n<\omega\rangle]$, so $A\in V[\langle\alpha_n\colon n<\omega\rangle]$.
\end{proof}

The above two lemmas gives the proof of $\mathbb{P}_{\vec{\mu}}$ yields minimal extensions.

\subsection{$\mathbb{P}_{\mathcal{D}}$ and diagonal $(*)$-iterated ultrapower}

A process similar to above is exhibited:

$\langle\gamma_n\colon n<\omega\rangle$ and $\mathcal{D}$ are as above. Given $(k,\alpha)\in\dom(\mathcal{D})$. Consider iterated ultrapower \[\mathcal{B}(k,\alpha)=(\langle M_n,j_{n,m}\colon n\leq m\leq\omega\rangle, \langle\kappa_n,U_n\colon n<\omega\rangle)\] satisfies:
\begin{itemize}
\item $M_0=V$, $\kappa_0=\gamma_k$, $U_0=\mathcal{D}(k,\alpha)$.
\item For each $n<\omega$, $\kappa_{n}=\gamma_{n+k}$ and $U_{n+1}=j_{0,n+1}(\mathcal{D})(n+k+1,\gamma_{n+k})$.
\end{itemize}

Similarly to proposition \ref{diedaixingzhi},

\begin{proposition}
\begin{itemize}
\item[(1)] $\mathcal{B}(k,\alpha)$ exists and is unique.
\item[(2)] For each $n$, $\crit(j_{n,n+1})=\gamma_{n+k}$. $\mathcal{B}(k,\alpha)$ is a diagonal $(*)$-iterated ultrapower.
\item[(3)] $\langle\gamma_k,\gamma_{k+1},...\rangle$ is a $j_{0,\omega}(\mathbb{P}_{\mathcal{D}}^{(k,\alpha)})$-generic sequence over $M_{\omega}$.
\item[(4)] $j_{0,\omega}(\gamma)=\gamma$. For $i<k$, $j_{0,\omega}(\gamma_i)=\gamma_i$; for $i\geq k$, $j_{0,\omega}(\gamma_i)=j_{i-k,i-k+1}(\gamma_i)$.
\end{itemize}
\end{proposition}

The following three lemmas gives another proof of $\mathbb{P}_{\mathcal{D}}$ yields minimal extensions, the proofs are similar to corresponding lemmas for $\mathbb{P}_{\vec{\mu}}$, so we omit them here.

\begin{lemma}
In the model $M_{\omega}[\langle\gamma_i\colon k\leq i<\omega\rangle]$, $a\colon\omega\rightarrow\gamma$ such that $\forall m\ (\gamma_{k+m}\leq a(m)<j_{0,\omega}(\gamma_{k+m}))$. Then $M_{\omega}[a]=M_{\omega}[\langle\gamma_i\colon k\leq i<\omega\rangle]$.
\end{lemma}

\begin{lemma}
Suppose $G$ is a $\mathbb{P}_{\mathcal{D}}$-generic filter over $V$, $g$ is the corresponding generic sequence. In $V[G]$, $a\colon\omega\rightarrow\gamma$ such that there is $m<\omega$, $\forall n\geq m(g(n)\leq a(n)<\gamma_n)$. Then $V[G]=V[a]$.
\end{lemma}

\begin{lemma}
Suppose $G$ is a $\mathbb{P}_{\mathcal{D}}$-generic filter over $V$. In $V[G]$, $A\subset\gamma$ and $A\notin V$. Let $g$ be the corresponding generic sequence of $G$. Then there is a sequence $\langle\alpha_n\colon n<\omega\rangle$ such that
\begin{itemize}
\item[(1)] For all $n$, $g(n)\leq\alpha_n<\gamma_n$.
\item[(2)] $V[A]=V[\langle\alpha_n\colon n<\omega\rangle]$.
\end{itemize}
\end{lemma}

\section{Forcing notion $\mathbb{Q}_{\mathcal{D}}$ and application in $\alpha$-recursion theory}

\subsection{The model $\mathcal{V}_{\gamma}$}

In this section, we use the same large cardinal assumption as the position we define $\mathbb{P}_{\mathcal{D}}$: $\langle\gamma_n\mid{n<\omega}\rangle$
is a strictly increasing sequence of infinite cardinals such that $\forall{n>0}$, $\gamma_n$
has at least $\gamma_{n-1}$ many normal measures. $\mathcal{D}$ is a list of normal measures as above. Let $\gamma$ be the supremum of $\langle\gamma_n\mid{n<\omega}\rangle$.

\begin{proposition}
For each $n<\omega$, $\gamma$ is $\Sigma_n$-admissible.
\end{proposition}
\begin{proof}
$\gamma=\lim_{n\rightarrow\omega}\gamma_n$. $\gamma$ is the limit of a sequence of infinite cardinals, so $\gamma$ is a singular cardinal whose cofinality is $\omega$.
$\gamma$ is a limit of Silver's indiscernibles, so
$\gamma$ is also a Silver's indiscernibles. Thus $\gamma$ is an inaccessible cardinal in $L$, so
\[L\models``L_{\gamma}\models\ZFC"\]So $L_{\gamma}\models\ZFC$.
So for each $n$, $\gamma$ is $\Sigma_n$-admissible.
\end{proof}

Let us define another reduction relation on subsets of $\gamma$ which will induce a new degree structure.
Every subset of $\gamma$ belongs to a
$\gamma$-degree, the foundation of defining $\gamma$-degree is the use of the model $L_{\gamma}$. Here we want to use another appropriate model to replace
$L_{\gamma}$. It is the model $\mathcal{V}_{\gamma}=(V_{\gamma},\in,f_1,f_2,f_3)$, where $f_i$'s are unary function symbols. $\mathcal{V}_{\gamma}$ satisfies:

\begin{itemize}
\item For $x\in\omega$, $f_1(x)=\gamma_x$; $x\notin\omega$, $f_1(x)=\emptyset$.
\item For $x$ is not an ordinal, $f_2(x)=\emptyset$. $f_2{\uh}\gamma\colon\gamma\rightarrow V_{\gamma}$ is a bijection, such that
\[\forall{x,y}\in{V_{\gamma}}(\rank(x)<\rank(y)\rightarrow f_2^{-1}(x)< f_2^{-1}(y)).\]
\item For $x$ is not an ordinal, $f_3(x)=\emptyset$. $f_3{\uh}\gamma\colon\gamma\rightarrow V_{\gamma}$ is an one-to-one, such that $f_3(0)=\mathcal{D}(0,0)$; $f_3{\uh}(\gamma_0{\setminus}\{0\})$ is an enumeration of $\{\mathcal{D}(1,\alpha)\colon\alpha<\gamma_0\}$; for $i<\omega$, $f_3{\uh}(\gamma_{i+1}{\setminus}\gamma_i)$ is an enumeration of $\{\mathcal{D}(i+2,\alpha)\colon\alpha<\gamma_{i+1}\}$.
\end{itemize}

Use $\vartriangleleft$ to denote the well-ordering of $V_{\gamma}$ induced by $f_2$, which has order type $\gamma$.
Now we use the model
$\mathcal{V}_{\gamma}$ to replace $(L_{\gamma},\in)$ in $\alpha$-recursion theory to define ``$V_{\gamma}$-degrees".

\subsection{$V_{\gamma}$-degrees}

If $K$ is an element of $V_{\gamma}$, we call $K$ a \emph{$V_{\gamma}$-finite set}. $V_{\gamma}$-finite set and $\gamma$-finite set have different properties:

\begin{proposition}
\begin{itemize}
\item[(1)] For $X\subset\gamma$, $X$ is
$V_{\gamma}$-finite iff $X$ is a bounded subset of $\gamma$.
\item[(2)] A subset of a $V_{\gamma}$-finite set is also a $V_{\gamma}$-finite set. Thus for any $T\subset{V_{\gamma}}$, any $V_{\gamma}$-finite set $x$, $T\cap{x}$ is a $V_{\gamma}$-finite set.
\end{itemize}
\end{proposition}

\begin{proof}
(1) $V_{\gamma}$ is the collection of all sets of rank less than $\gamma$, so $X\in{V_{\gamma}}$ iff $X$ is bounded.
(2) A subset of element of $V_{\gamma}$ is also an element of $V_{\gamma}$. $T\cap{x}\subset{x}$, so $T\cap{x}\in{V_{\gamma}}$.
\end{proof}

This proposition tells us, using the model $\mathcal{V}_{\gamma}$ instead of $(L_{\gamma},\in)$, we need not consider the regularity of a set, since every subset of $\gamma$ is ``regular" in the sense of $\mathcal{V}_{\gamma}$.

\begin{definition}
We call $A\subset\gamma$ is a \emph{$V_{\gamma}$-RE set}, if there exists a
$\Sigma_1$-formula $\psi(x)$ in the language of $\mathcal{V}_{\gamma}$ with $V_{\gamma}$-finite sets as parameters, such that $A=\{a\in{{\gamma}}\colon {\mathcal{V}_{\gamma}}\models\psi(a)\}$.
\end{definition}

\begin{remark}
Note that $\mathcal{V}_{\gamma}$ has more function symbols than $(V_{\gamma},\in)$, so a $\Delta_0$-formula for $\mathcal{V}_{\gamma}$ has a little different meaning: $\exists x\in t\ (...)$, where $t$ is a term not just a variable symbol. $\Sigma_1$-formula is those formulas of the form
$\exists{x}\phi$, where $\phi$ is a $\Delta_0$-formula.
\end{remark}

For $n<\omega$, define a first-order model \[\mathcal{V}_{\gamma,n}=(V_{\gamma_n},\in,\gamma_0,...,\gamma_{n-1},f_2{\uh}\gamma_n,f_3{\uh}\gamma_{n-1}).\]
Because $f_2{\uh}\gamma_n\colon\gamma_n\rightarrow V_{\gamma_n}$, so this is well-defined. For each $n<\omega$, the model $\mathcal{V}_{\gamma,n}\in V_{\gamma}$ and for $\varphi$, a formula in the language of $\mathcal{V}_{\gamma,n}$, ``$\mathcal{V}_{\gamma,n}\vDash\varphi$" is equivalent to a $\Delta_0$-formula in $\mathcal{V}_{\gamma}$. Moreover, we know if $\psi$ is a $\Sigma_1$-formula in the language of $\mathcal{V}_{\gamma}$, then $\mathcal{V}_{\gamma}\vDash\psi$ is equivalent to a $\Sigma_1$-formula ``there is $n$, $\psi$ is true in $\mathcal{V}_{\gamma,n}$." So we have the following which is similar to $\alpha$-recursion theory:

\begin{proposition}\label{universalreV}
There exists a $V_{\gamma}$-RE set $W$ such that
\begin{itemize}
\item[(i)] For each
$e<\gamma$, $W_{e}=\{x\colon J(e,x)\in{W}\}$ is a
$V_{\gamma}$-RE set.
\item[(ii)] If $B$ is a $V_{\gamma}$-RE set, then there is $e<\gamma$ such that
$B=W_{e}=\{x\colon J(e,x)\in{W}\}$.
\end{itemize}
\end{proposition}

\begin{proposition}
Suppose $A\subset\gamma$ and $A$ is $\gamma$-RE, then $A$ is $V_{\gamma}$-RE.
\end{proposition}

\begin{proof}
This follows from the fact that $L_{\gamma}$ is $\Sigma_1$-definable in $\mathcal{V}_{\gamma}$.
\end{proof}

Fix a $V_{\gamma}$-RE set $W$ which has the property stated in Lemma \ref{universalreV}.
For each $e<\gamma$, \[W_{e}=\{x\colon J(e,x)\in{W}\}\]
Now define the reduction relation between subsets of $\gamma$. As in notations in Section 1, given $P,X\subset\gamma$, $P^{X}=\{a\in\gamma\colon\exists x\in X\ J(a,x)\in P\}$. For $A\subset\gamma$, define \[N(A)=\{J(x,y)\colon f_2(x)\subset A\ \&\ f_2(y)\subset\gamma{\setminus}A\}.\] $P^{X}$ and $N(A)$ are both subsets of $\gamma$.

\begin{notation}
$A,B\subset\gamma$, $e<\gamma$, we say ``\emph{$B$ can enumerate $A$ using $W_e$}" iff $A=W_e^{B}$. We say ``\emph{$B$ can enumerate $A$}" iff there is $e<\gamma$, $A=W_e^{B}$. A pair $(K,L)$ such that $K,L$ are $V_{\gamma}$-finite and $K\subset A\ \&\ L\subset\gamma{\setminus}A$, is called a \emph{$V_{\gamma}$-finite information} for $A$.
\end{notation}

\begin{proposition}
\begin{itemize}
\item[(1)] $A$ can enumerate $A$.
\item[(2)] If $B$ can enumerated $A$ and $C$ can enumerated $B$, then $C$ can enumerated $A$.
\end{itemize}
\end{proposition}

\begin{proof}
(1) Because $\{J(x,x)\colon x<\gamma\}$ is $V_{\gamma}$-RE set.

(2) $W,Y$ are $V_{\gamma}$-RE sets such that $A=W^{B}$ and $B=Y^{C}$.
\begin{equation*}
\begin{split}
x\in A\ &\leftrightarrow\ \exists t\ (t\in B\ \&\ J(x,t)\in W)\\
        &\leftrightarrow\ \exists t\ (\exists s\ (s\in C\ \&\ J(t,s)\in Y)\ \&\ J(x,t)\in W)\\
        &\leftrightarrow\ \exists s\ (s\in C\ \&\ \exists t(J(t,s)\in Y\ \&\ J(x,t)\in W))\\
        &\leftrightarrow\ \exists s\ (s\in C\ \&\ J(x,s)\in P).
\end{split}
\end{equation*}
where $P$ is a $V_{\gamma}$-RE set defined by $k\in P\ \leftrightarrow\ \exists t(J(t,(k)_1)\in Y\ \&\ J((k)_0,t)\in W)$.
\end{proof}

\begin{definition}\label{def:vgammadegree}
For $A,B\subset\gamma$, define \emph{$A\leq_{V_{\gamma}}B$} iff there is a $V_{\gamma}$-RE set $P$ such that $N(A)=P^{N(B)}$. This is just $N(B)$ can enumerate $N(A)$.
\end{definition}

From the proposition above, the relation $\leq_{V_{\gamma}}$ is reflexive and transitive.
For $A,B\subset\gamma$, $A\equiv_{V_{\gamma}}B$ iff
$A\leq_{V_{\gamma}}B$ and $B\leq_{V_{\gamma}}A$.
$\equiv_{V_{\gamma}}$ is an equivalence relation, every equivalence class is called a
\emph{$V_{\gamma}$-degree}. If $A\equiv_{V_{\gamma}}\emptyset$, then we say $A$ is
\emph{$V_{\gamma}$-recursive set}. The $V_{\gamma}$-degree consists of $V_{\gamma}$-recursive set is denoted by
$\mathbf{0}$. If $A$ is a $V_{\gamma}$-finite set, then $A$ is a
$V_{\gamma}$-recursive set. $A>_{V_{\gamma}}B$ iff
$A\geq_{V_{\gamma}}B$ and $\neg(B\geq_{V_{\gamma}}A)$. The following result is natural from our assumption:

\begin{lemma}\label{xuxing}
For each $B\subset\gamma$, there is $X=\{x_i\mid{0<i<\omega}\}$ such that
$\forall{0<i<\omega}(x_i<\gamma_i)$, and $X\equiv_{V_{\gamma}}B$.
\end{lemma}

\begin{proof}
If $n<\omega$, then
$B\cap\gamma_n\subset\gamma_n$, so $B\cap\gamma_n\in{V_{\gamma_{n+1}}}$,
so $f_2^{-1}(B\cap\gamma_n)<\gamma_{n+1}$. If $i>0$, define
$x_i=f_2^{-1}(B\cap\gamma_{i-1})$.
$B\geq_{V_{\gamma}}X$ is obvious. Check
$X\geq_{V_{\gamma}}B$. Use $x_i$ to enumerate all ordered pairs consists of all subsets of $B\cap\gamma_{i-1}$
and all subsets of $\gamma_{i-1}{\setminus}B$.
\end{proof}

Suppose $A\subset\gamma$. If $A$ is a $V_{\gamma}$-recursive set, then $A$ and $\gamma{\setminus}A$ are both $V_{\gamma}$-RE set. But converse statement is not true. $A$ is
$V_{\gamma}$-recursive means $N(A)$ is $V_{\gamma}$-RE. To enumerate all $V_{\gamma}$-finite information of $A$, we should justify whether $x\in{K}$ for every $V_{\gamma}$-finite set $K$. So this cannot be completed in $V_{\gamma}$-finite time. The next lemma tells us, the structure of  $V_{\gamma}$-degrees is not trivial:

\begin{lemma}\label{bukejisuan}
There exists $A\subset\gamma$ such that $A>_{V_{\gamma}}\emptyset$.
\end{lemma}

\begin{proof}
Fix $\langle{f_e}\mid{e<\gamma}\rangle$ is the enumeration of all $\Sigma_1$-definable partial function (from $\gamma$ to $\gamma$) in $\mathcal{V}_{\gamma}$.
Let
$A=\{x<\gamma\mid{f_{x}(x)}$ has definition$\}$. Suppose $A$ is
$V_{\gamma}$-recursive. Define the function $F$:
\[
F(x)=\left\{
\begin{array}{ll}
f_{x}(x)+1  &\mbox{ $f_{x}(x)$ has definition}.\\
0         &\mbox{ $f_{x}(x)$ no definition}.
\end{array}
\right.
\]
Since $A$ is $V_{\gamma}$-recursive, So $F$ is $\Sigma_1$-definable.
But for each $e<\gamma$, $F\neq{f_{e}}$. A contradiction.
So $A>_{V_{\gamma}}\emptyset$.
\end{proof}

The effect of $\Sigma_1$-admissibility in $\alpha$-recursion theory is to define a function recursively, see \cite{16}, in particular, $\Sigma_1$-boundedness is important in $\alpha$-recursion theory. In our environment, $\Sigma_1$-boundedness fails, but given a function $f$ from $\omega$ to $\omega$, we can still construct a function bounded by $\gamma_{f(n)}$ at the stage $n$. The proof is clear:

\begin{proposition}\label{induction2}
$f\colon\omega\rightarrow\omega$ is a recursive function. $I\colon V_{\gamma}\rightarrow V_{\gamma}$ such that if $\sigma$ is a sequence with length less than $\gamma_n$, then $I(\sigma)\in V_{\gamma_{f(n)}}$. Then there is unique function $g$ such that $g(\alpha)=I(g{\uh}\alpha)$ and the graph of $g$ is $\Sigma_1$-definable in $\mathcal{V}_{\gamma}$.
\end{proposition}

\subsection{Isomorphic to a cone of $\gamma$-degrees}

Let us use a subset $W$ of $\gamma$ to code the model $\mathcal{V}_{\gamma}$:
Define
\begin{equation*}
\begin{split}
W_0=\{J&(\alpha,\beta)\colon f_2(\alpha)\in f_2(\beta)\}.\\
W_1=\{J&(\alpha,\beta)\colon f_1(f_2(\alpha))=f_2(\beta)\}.\\
W_2=\{J&(\alpha,\beta)\colon f_3(f_2(\alpha))=f_2(\beta)\}.
\end{split}
\end{equation*}
Let $W=W_0\oplus W_1\oplus W_2$. By definition, we know

\begin{proposition}
$W$ is $V_{\gamma}$-RE.
\end{proposition}

Because $f_2{\uh}\gamma\colon\gamma\rightarrow V_{\gamma}$ is a bijection, so $(\gamma,W_0,W_1,<,W_2)$ is isomorphic to the model $\mathcal{V}_{\gamma}$.

Let $\mathbf{w}$ be the $\gamma$-degree which $W$ belongs. Let $\Cone(\mathbf{w})$ be the set of all $\gamma$-degrees stronger than $\mathbf{w}$ and
$\mathcal{S}$ be the set of all $V_{\gamma}$-degrees.

\begin{theorem}\label{tonggou}
$(\Cone(\mathbf{w}),\geq_{\gamma})$ and
$(\mathcal{S},\geq_{V_{\gamma}})$ are isomorphic.
\end{theorem}

The proof of this theorem follows from the following lemma \ref{tonggou1} and lemma \ref{tonggou2}.

\begin{proposition}\label{xuxing2}
\begin{itemize}
\item[(1)] If $P\subset L_{\gamma}$, define $P_1=\{\alpha<\gamma\colon K_{\alpha}\in P\}\subset\gamma$, $P_2=\{\beta<\gamma\colon f_2(\beta)\in P\}\subset\gamma$. Then $P_1$ can enumerate $P_2$.
\item[(2)] Suppose $X\subset\gamma$, $X$ is cofinal in $\gamma$ and the order type of $X$ is $\omega$. Then in the model $\mathcal{V}_{\gamma}$, $\NL(X)$ can enumerate $N(X)$.
\end{itemize}
\end{proposition}

\begin{proof}
(1) $\{J(f_2^{-1}(K_{\alpha}),\alpha)\colon\alpha<\gamma\}$ is a $V_{\gamma}$-RE set.

(2) From (1) and the fact that if $X$ of order type $\omega$ and cofinal, then for every subset $t$, $t$ is finite iff $t$ is $\gamma$-finite iff $t$ is $V_{\gamma}$-finite.
\end{proof}

Given $B\subset\gamma$, define $B'$ as follows: $B\cap\gamma_0$ is an element of
$V_{\gamma_1}$, so $\pi^{-1}(B\cap\gamma_0)<\gamma_1$, define
$x_i$ is $\pi^{-1}(B\cap\gamma_{i-1})$, then define
$B'=\{x_i\mid{0}<i<\omega\}$. By lemma \ref{xuxing}, $B\equiv_{V_{\gamma}}B'$.

\begin{lemma}\label{tonggou1}
Suppose $A,B\subset\gamma$, if $A\oplus{W}\geq_{\gamma}B$, then
$A\geq_{V_{\gamma}}B$.
\end{lemma}

\begin{proof}
The following enumerations are all in the model $\mathcal{V}_{\gamma}$.
From the definition of $B'$, $A\oplus{W}\geq_{\gamma}B$ iff
$A\oplus{W}\geq_{\gamma}B'$. $A\geq_{V_{\gamma}}B$
iff $A\geq_{V_{\gamma}}B'$. So without loss of generality, we can assume $B$ is equal to
$B'$, i.e. the order type of $B$ is at most $\omega$.

We use the $V_{\gamma}$-finite information of $A$ to enumerate the $V_{\gamma}$-finite information of $B$ in the following steps:

\noindent\underline{\emph{Step 1.}} Because $W$ is a $V_{\gamma}$-RE set, one can use the $V_{\gamma}$-finite information of $A$, one can enumerate all $V_{\gamma}$-finite information of $A\oplus{W}$.

\noindent\underline{\emph{Step 2.}} Suppose $(K,L)$ is a $V_{\gamma}$-finite information of $A\oplus{W}$, then $(K,L)$ is a $L_{\gamma}$-finite information of $A\oplus{W}$ iff $K,L\in{L_{\gamma}}$.
Thus use the $V_{\gamma}$-finite information of $A\oplus{W}$, one can enumerate all $L_{\gamma}$-finite information of $A\oplus{W}$.

\noindent\underline{\emph{Step 3.}} Because $A\oplus{W}\geq_{\gamma}B$, so there is ${\gamma}$-RE set $U_{e}$, $\NL(B)=U_e^{\NL(A\oplus W)}$, because $U_e$ is also a $V_{\gamma}$-RE set, so use the $L_{\gamma}$-finite information of $A\oplus{W}$, one can enumerate all $L_{\gamma}$-finite information of $B$ in $\mathcal{V}_{\gamma}$.

\noindent\underline{\emph{Step 4.}} Because $B$ is of order type at mose $\omega$, by proposition \ref{xuxing2},
use the $L_{\gamma}$-finite information of $B$, one can enumerate all $V_{\gamma}$-finite information of $B$.

So we have enumerated all $V_{\gamma}$-finite information of $B$, using the $V_{\gamma}$-finite information of $A$.
\end{proof}

\begin{lemma}\label{tonggou2}
Suppose $A,B\subset\gamma$, if $A\geq_{V_{\gamma}}B$, then
$A\oplus{W}\geq_{\gamma}B$.
\end{lemma}

\begin{proof}
Enumerations in this proof are all in the sense of $L_{\gamma}$.
We want to use the $L_{\gamma}$-finite information of $A\oplus{W}$ to enumerate the $L_{\gamma}$-finite information of $B$.

Note that the following fact: define
\[T=\{J(\alpha,\beta)\colon\alpha<\gamma\ \&\ \beta\mbox{ is the }\alpha\mbox{-th ordinal in }(\gamma,W_0,W_1,<,W_2)\}.\]
\[T_1=\{J(\alpha,\beta)\colon\alpha<\gamma\ \&\ \beta\mbox{ is the }\alpha\mbox{-th constructible set in }(\gamma,W_0,W_1,<,W_2)\}.\]
Then $\NL(T)$ and $\NL(T_1)$ can be enumerated by $\NL(W)$ in the sense of $L_{\gamma}$.

The following is the steps:

\noindent\underline{\emph{Step 1.}} Enumerate $f_2^{-1}[N(A)]$ from $\NL(A)$ and $\NL(W)$.

For $n<\omega$, define $C_n=\{J(\beta_1,\beta_2)\colon\forall x<\gamma_n\ (J(x,\beta_1)\in W_0\rightarrow J(x,\beta_2)\in W_0)\}$. Then $C_n$ is a $\gamma$-RE set. Also we have
\begin{equation*}
\begin{split}
f_2(\beta)\subset A\ &\leftrightarrow\ \exists n<\omega\ J(\beta,A\cap\gamma_n)\in C_n.\\
f_2(\beta)\subset \gamma{\setminus}A\ &\leftrightarrow\ \exists n<\omega\ J(\beta,\gamma_n{\setminus}A)\in C_n.
\end{split}
\end{equation*}
So $f_2^{-1}[\{x\in V_{\gamma}\colon x\subset A\}]$ and $f_2^{-1}[\{x\in V_{\gamma}\colon x\subset \gamma{\setminus}A\}]$ can be enumerated from $\NL(A)$ and $\NL(W)$.

\noindent\underline{\emph{Step 2.}} Enumerate $f_2^{-1}[N(B)]$ from $f_2^{-1}[N(A)]$ and $\NL(W)$.

Since $A\geq_{V_{\gamma}}B$, so this is clear.

\noindent\underline{\emph{Step 3.}} Enumerate $f_2^{-1}[\NL(B)]$ from $f_2^{-1}[N(B)]$ and $\NL(W)$.

Since $L_{\gamma}$ is definable in $V_{\gamma}$ effectively, so from $f_2^{-1}[N(B)]$ and $\NL(W)$, $f_2^{-1}[\NL(B)]$ is enumerable.

\noindent\underline{\emph{Step 4.}} Enumerate $\NL(B)$ from $f_2^{-1}[\NL(B)]$ and $\NL(W)$.

$T_1$ is the set defined at the beginning of this proof. Define $\tilde{T_1}=\{J(\beta,\alpha)\colon J(\alpha,\beta)\in T_1\}$. Then $\tilde{T_1}^{f_2^{-1}[\NL(B)]}=\NL(B)$. But $\tilde{T_1}$ can be enumerated from $\NL(W)$. So $\NL(B)$ is enumerable from $f_2^{-1}[\NL(B)]$ and $\NL(W)$.

So we use the $L_{\gamma}$-finite information of $A\oplus{W}$ to enumerate all $L_{\gamma}$-finite information of $B$.
\end{proof}

\subsection{The forcing $\mathbb{Q}_{\mathcal{D}}$}

We use a variant of $\mathbb{P}_{\mathcal{D}}$ to handle $V_{\gamma}$-degrees.

\begin{proposition}
There is a function $g\colon V_{\gamma}\rightarrow V_{\gamma}$ such that
\begin{itemize}
\item[(a)] $\{(x,y)\colon g(x)=y\}$ is $\Sigma_1$-definable in $V_{\gamma}$.
\item[(b)] If $x\in\gamma$, $g(x)\in f_3(x)$.
\item[(c)] If $x,y\in\gamma$ and $x\neq y$, then $g(x)\cap g(y)=\emptyset$.
\end{itemize}
\end{proposition}

\begin{proof}
Go the same proof as proposition \ref{bujiao} in the model $\mathcal{V}_{\gamma}$. Proposition \ref{induction2} is used here. Pick every object by the $\vartriangleleft$-least.
\end{proof}

Fix such a $g$ in this proposition. Let us define our forcing notion $\mathbb{Q}_{\mathcal{D}}$.

\begin{definition}
The forcing conditions in $\mathbb{Q}_{\mathcal{D}}$ is of form $p=(D_p,<_p)$ satisfies:
\begin{itemize}
\item[(i)] $p$ is a tree of height $\omega$, $D_p\subset\gamma$, $<_p$ is a tree relation on $D_p$ and \[\forall x,y\in D_p\ (x<_p y\rightarrow x<y).\]
\item[(ii)] For $x\in D_p$, $h(x)$ is the height of $x$ in $p$. For all $n$, if $h(x)=n$, then $x<\gamma_n$.
\item[(iii)] $\stem(p)$ is the $<_p$-largest element of $D_p$ which is $<_p$-comparable with all elements of $D_p$. If $\alpha\in D_p$ such that $\alpha\geq_p\stem(p)$, \[A_{\alpha}=\{x\in D_p\colon h(x)=h(\alpha)+1\ \&\ \alpha<_p x\},\] then $A_{\alpha}\subset g(\alpha)$ and $A_{\alpha}\in f_3(\alpha)$.
\item[(iv)] $D_p$ and $\{J(x,y)\colon x,y\in D_p\ \&\ x<_p y\}$ are $V_{\gamma}$-RE sets.
\end{itemize}

Define $p$ is stronger than $q$ iff $D_p\subset D_q$.
\end{definition}

Given $D_p$ for some $p\in\mathbb{Q}_{\mathcal{D}}$, $<_p$ is completely determined. So we usually use $p$ to denote $D_p$. For $p,q\in\mathbb{Q}_{\mathcal{D}}$, $p$ is a \emph{pure extension} of $q$, if $p\subset q$ and $\stem(p)=\stem(q)$. Note that $p\in\mathbb{Q}_{\mathcal{D}}$, then $p\notin V_{\gamma}$, but $p$ is $\Sigma_1$-definable in $\mathcal{V}_{\gamma}$. From the definitions, we have

\begin{proposition}\label{619}
\begin{itemize}
\item[(1)] There is a $V_{\gamma}$-RE set $R$, such that
\begin{itemize}
\item For each $e<\gamma$, $R_e=\{x<\gamma\colon J(x,e)\in R\}\in\mathbb{Q}_{\mathcal{D}}$.
\item If $p\in\mathbb{Q}_{\mathcal{D}}$, then there is $e<\gamma$, $p=R_e$.
\end{itemize}
\item[(2)] There is $V_{\gamma}$-RE set $S$, such that
\begin{itemize}
\item For each $e<\gamma$, $x$ is the least element of $S_e$, then $p_e=S_e{\setminus}\{x\}\in\mathbb{Q}_{\mathcal{D}}$ and $x=\stem(p)$.
\item If $p\in\mathbb{Q}_{\mathcal{D}}$, then there is $e<\gamma$, $p=p_e$.
\end{itemize}
Thus ``$\stem(p)=x$" is a $\Sigma_1$-formula in $\mathcal{V}_{\gamma}$.
\item[(3)] ``$p$ is a pure extension of $q$" is a $\Sigma_1$-formula in $\mathcal{V}_{\gamma}$.
\end{itemize}
\end{proposition}

Fix a $V_{\gamma}$-RE set described in proposition above (1). We want to obtain a branch of trees in some ``generic" subset of $\mathbb{Q}_{\mathcal{D}}$ to satisfies that it has minimal $V_{\gamma}$-degree.

\subsection{Construction of minimal $V_{\gamma}$-degrees}

Suppose $\mathbf{a}$ is a $V_{\gamma}$-degree. $\mathbf{a}$ is called a \emph{minimal
$V_{\gamma}$-degree} iff $\mathbf{a}>_{V_{\gamma}}\mathbf{0}$ and there is no $\mathbf{c}$ such that
$\mathbf{0}<_{V_{\gamma}}\mathbf{c}<_{V_{\gamma}}\mathbf{a}$.

Note that the sequence $\langle E_n^{\alpha}\colon n<\omega\rangle$, $p^{(x)}$, $p^A$, $p(n)$ are all defined as above. For every forcing condition $p\in\mathbb{Q}_{\mathcal{D}}$, $p(n)\in E_n^{p}$ and if $q$ is a subtree of $p$, then $q\in\mathbb{Q}_{\mathcal{D}}$ iff for all $n$, $q(n)\in E_n^q$ and $q$ is $V_{\gamma}$-RE. Corollary \ref{hanshufenlei2} for classification of functions are also valid here.

\begin{lemma}
Suppose $e,x<\gamma$. There is a $V_{\gamma}$-RE set $R'$ such that for each $t<\gamma$, $R'_t\in\mathbb{Q}_{\mathcal{D}}$ and $R'_t$ is a pure extension of $R_t$ such that $(a)$ or $(b)$ holds:
\begin{itemize}
\item[(a)] If $G$ is a branch of $R'_t$, $x\notin W_e^{N(G)}$.
\item[(b)] If $G$ is a branch of $R'_t$, $x\in W_e^{N(G)}$.
\end{itemize}
\end{lemma}

\begin{proof}
Let us define an operator, input a forcing condition $p$ of $\mathbb{Q}_{\mathcal{D}}$, and output a pure extension $q$.
Begin from $p$. For $\alpha\in p$ such that $\alpha>_p\stem(p)$ and $\gamma_{n-1}<\alpha<\gamma_n$. Define
\[B(\alpha)=\{J(z,y)\colon f_2(z)\subset\{s\in p\colon s\leq_p\alpha\}\ \&\ f_2(y)\subset\gamma_n{\setminus}\{s\in p\colon s\leq_p\alpha\}\}.\]
Then $B(\alpha)\in V_{\gamma}$. Suppose $G$ is a branch of $p$, then we have $W_e^{N(G)}=\bigcup_{n=1}^{\infty}B(G(n))$.

Ask a question: is there $n<\omega$ such that $\{\alpha\in p(n)\colon x\in W_e^{B(\alpha)}\}\in E_n^p$?

This question is $V_{\gamma}$-RE. The reason is: $\varphi(w,\alpha)$ is the formula ``$w\in W_e^{B(\alpha)}$" is $\Sigma_1$ in $\mathcal{V}_{\gamma}$.

\noindent(1) If the answer is ``yes". Use the set $\{\alpha\in p(n)\colon x\in W_e^{B(\alpha)}\}$ to shrink $p(n)$, i.e. let $q=p^{\{\alpha\in p(n)\colon x\in W_e^{B(\alpha)}\}}$. Then $q$ is a pure extension of $p$ and $q$ is also a forcing condition of $\mathbb{Q}_{\mathcal{D}}$. If $G$ is a branch of $q$, then $G(n+|\stem(q)|-1)\in q(n)$, so \[x\in W_e^{B(G(n+|\stem(q)|-1))}\subset W_e^{N(G)}.\]

\noindent(2) If the answer is ``no". Then there is a pure extension $q$, such that if $\alpha\in q$, then $x\notin W_e^{B(\alpha)}$. Suppose If $G$ is a branch of $q$, since $W_e^{N(G)}=\bigcup_{n=1}^{\infty}B(G(n))$, so $x\notin W_e^{N(G)}$.
\end{proof}

\begin{lemma}
Suppose $e<\gamma$. There is a $V_{\gamma}$-RE set $\Omega$ such that for each $t<\gamma$, $\Omega_t\in\mathbb{Q}_{\mathcal{D}}$ and $\Omega_t$ is an extension of $R_t$ such that if $G$ is a branch of $\Omega_t$, $A\subset\gamma$ and $N(A)=W_e^{N(G)}$, then $(a)$ or $(b)$ holds:
\begin{itemize}
\item[(a)] $A$ is $V_{\gamma}$-recursive.
\item[(b)] $A\geq_{V_{\gamma}}G$.
\end{itemize}
\end{lemma}

\begin{proof}
Let us define an operator, input a forcing condition $p$ of $\mathbb{Q}_{\mathcal{D}}$, and output an extension $q$.
Begin from $p$.

\noindent\underline{\emph{Step 1.}} Shrink $p$ to a pure extension $p_1$ such that for all $\alpha\in p_1$ such that $\alpha>_{p_1}\stem(p_1)$ and $\gamma_{n-1}<\alpha<\gamma_n$, for all $x<\gamma_n$, then (a) or (b) happens:
\begin{itemize}
\item[(a)] if $G$ is a branch of $p_1^{(\alpha)}$, then $x\in W_e^{N(G)}$.
\item[(b)] if $G$ is a branch of $p_1^{(\alpha)}$, then $x\notin W_e^{N(G)}$.
\end{itemize}
The reason is: for $\gamma_n$ many $\alpha$, such that $\gamma_{n-1}<\alpha<\gamma_n$, intersection of $p^{(\alpha)}$ is also a forcing condition of $\mathbb{Q}_{\mathcal{D}}$; for one $x<\gamma_n$, the lemma above give us the method to shrink the condition.

This step output $p_1$ and the function $F$ with domain $p_1$, the value on $\alpha$ is the subset of $\gamma_n$ consists those $x$ which is in $W_e^{N(G)}$ for every branch of $p_1^{\alpha}$. Then $p_1\in\mathbb{Q}_{\mathcal{D}}$ and the graph of $F$ is $V_{\gamma}$-RE.

\noindent\underline{\emph{Step 2.}} Ask a question: is there $n<\omega$, such that
\[\{\alpha\in p_1(n)\colon \exists T\subset\gamma_{n+|\stem(p_1)|-1}\ F(\alpha)=\{J(x,y)\colon f_2(x)\subset T\mathrel{\&} f_2(y)\subset\gamma_n{\setminus}T\}\}\notin E_n^{p_1}?\]

\noindent(1) If the answer is ``yes". Then shrink $p_1$ to $p_2$ using the set in question. Then for all $G$, a branch of $p_2$, $W_e^{N(G)}$ is not $N(A)$ for some $A\subset\gamma$. Then our operator stops.

\noindent(2) If the answer is ``no". Then shrink $p_1$ to $p_2$ using the set in question. Output the condition $p_2$ and a function $F_1$ on $p_2$, $F_1(\alpha)$ is the $0$-$1$ sequence of length $\gamma_n$ which is determined by $p_2^{(\alpha)}$. Then go to the next step.

\noindent\underline{\emph{Step 3.}}
Given a forcing condition $p_2$ as the output of step 2, ask a question: is there $n<\omega$, such that $F_1{\uh}p_2(n)$ is not $E_n^{p_2}$-equivalent to a constant function?

\noindent(1) If the answer is ``yes". Then by the principle of classification of functions, there is $0<n'\leq n$, $F_1$ is an one-to-one function on $p_2(n')$. shrink $p_2$ to $p_3$ by the set in the question. Output $p_3$, $n'$ and $F_1{\uh}p_2(n')$.

\noindent(2) If the answer is ``no". Then shrink $p_2$ to $p_3$ by the set in the question. For each $G$, a branch of $p_3$, $W_e^{N(G)}$ is $N(A)$ for some $V_{\gamma}$-recursive set $A$. The operator stops.

And then repeat this process for $p_3^{(x)}$ for all $x\in p_3(n')$. Inductively, we have defined the operation at step 3.
We have two cases:

\noindent\emph{Case 1.} Step 3 does not stop at any stage.
Then we obtain a pure extension of $p_2$ and an one-to-one function, for any branch $G$ of $p_2$, $A\geq_{V_{\gamma}}G$.

\noindent\emph{Case 2.} Step 3 stops at some stage.
Then there is $V_{\gamma}$-recursive set $A$ such that $W_e^{N(G)}=N(A)$.

By the induction principle in $\mathcal{V}_{\gamma}$, the output of step 3 is also $V_{\gamma}$-RE. So we complete the proof of this lemma.
\end{proof}

Use $\Omega^{(e)}$ to denote the $V_{\gamma}$-RE set described in the above lemma.

\begin{lemma}
Suppose $e<\gamma$. There is a $V_{\gamma}$-RE set $\Theta$ such that for each $t<\gamma$, $\Theta_t\in\mathbb{Q}_{\mathcal{D}}$ and $\Theta_t$ is a pure extension of $R_t$ such that if $G$ is a branch of $\Theta_t$ then $N(G)\neq W_e$.
\end{lemma}

\begin{proof}
We will define an operator on forcing conditions in $\mathbb{Q}_{\mathcal{D}}$:
Begin from $p\in\mathbb{Q}_{\mathcal{D}}$, let us define a pure extension of $p$.

Ask the question: is there $n<\omega$, such that \[\{\alpha\in p(n)\colon \exists x\in W_e\ f_2^{-1}[\{t\in p\colon t\leq_p \alpha\}]=(x)_0\}\notin E_n^{p}\ ?\]

\noindent(1) If the answer is ``yes", let $n$ is the least witness. Then we obtain a pure extension $q$ by shrinking $p$ on the $n$-th level by the set in the question. Then $W_e$ is uncountable. But if $G$ is a branch of $q$, then $N(G)$ is countable, so $N(G)\neq W_e$.

\noindent(2) If the answer is ``no", then we obtain a pure extension $q$ by shrinking $p$ on the $n$-th level by complement of the set in the question. If $G$ is a branch of $q$, then $W_e\neq N(G)$.

This question is $V_{\gamma}$-RE. The result $q$ is also a forcing condition of $\mathbb{Q}_{\mathcal{D}}$. Thus we can define the $V_{\gamma}$-RE set $\Theta$.
\end{proof}

Use $\Theta^{(e)}$ to denote the $V_{\gamma}$-RE set described in the above lemma.

\begin{theorem}\label{jixiao}
There exists minimal $V_{\gamma}$-degree.
\end{theorem}

\begin{proof}
Let us construct a sequence $G\colon\omega{\setminus}\{0\}\rightarrow\gamma$.
Begin from condition $p_0$. $p_0$ is a condition of $\mathbb{Q}_{\mathcal{D}}$ such that
$h(\stem(p_0))=0$. Define $G(0)=\stem(p_0)$.

$p_0=R_{x_0}$ ($R_t$ is defined in Proposition \ref{619}).
$\langle(\Omega^{(e)}_{x_0},\Theta^{(e)}_{x_0})\mid{e<\gamma_0}\rangle$
is $\Sigma_1$-definable in $\mathcal{V}_{\gamma}$. Since $p_0$ has
$\gamma_1$-completeness(this means less than $\gamma_1$ many pure extensions of $p_0$ can be fused to a stronger pure extension of $p_0$.)
Thus \[\tilde{p_0}=\bigcap_{e<\gamma_0}\Omega^{(e)}_{x_0}\cap\bigcap_{e<\gamma_0}\Theta^{(e)}_{x_0}\]
is also a forcing condition of $\mathbb{Q}_{\mathcal{D}}$. Define $G(1)$ is the least element of
$\tilde{p_0}(1)$. Define $p_1=\tilde{p_0}^{G(1)}$.

Since $p_1$ has $\gamma_2$-completeness, so we can continue this process $\omega$ times. So we construct $G$, from two lemmas above, we know $G$ has minimal $V_{\gamma}$-degree.
\end{proof}

The following theorem is our conclusion, it is also the solution of question \ref{question}.

\begin{theorem}\label{thm:zuizhong}
Suppose $\langle\gamma_n\colon n<\omega\rangle$  is a strictly increasing sequence of measurable cardinals such that for each $n>0$,  $\gamma_n$ carries at least $\gamma_{n-1}$-many normal measures. Let $\gamma=\sup\{\gamma_n\colon n<\omega\}$. 
Then there is an $A\subset\gamma$ such that
\begin{itemize}
\item[(a)] $(L_{\gamma},\in,A)$ is not admissible.
\item[(b)] The $\gamma$-degree that contains $A$ has a minimal cover.
\end{itemize}
\end{theorem}

\begin{proof}
This follows from Theorem \ref{tonggou} and Theorem \ref{jixiao}.
\end{proof}
\begin{remark}
  Such sequence of measure cardinals can be found in an inner model for
  $o(\kappa)=\kappa$. Generalized degree structures at uncountable cardinals
  (especially of countable cofinality) are discussed in~\cite{ShiWoodin2012}.
  This result is a critical stage of their big picture.
\end{remark}

\subsection*{Acknowledgements}
During
the preparation of this paper, Andreas Blass, Moti Gitik answered my questions regarding ultrafilters, Yang Yue, Shi Xianghui and Yu Liang spent time with me for many stimulating discussions, I would like to thank them all.

\end{document}